\documentclass[a4paper, 10pt, reqno]{amsart}

\usepackage[margin=1in]{geometry}  
\usepackage{graphicx}               
\usepackage{amsmath}              
\usepackage{amsfonts}              
\usepackage{bm}                        
\usepackage{amsthm}                
\usepackage{mathtools}             
\usepackage{esint}                     
\usepackage{color}
\usepackage{amssymb}             
\usepackage[colorlinks]{hyperref}
\usepackage[noabbrev]{cleveref}
\usepackage{tikz}                       
\usetikzlibrary{patterns}              
\usetikzlibrary{fadings}
\usetikzlibrary{decorations.markings}
\usepackage{subcaption}
\usepackage[english]{babel}
\usepackage[utf8]{inputenc}
\usepackage[T1]{fontenc}
\usepackage{soul}                     
\usepackage{enumitem}

\newtheorem{thm}{Theorem}[section]
\newtheorem*{thm*}{Theorem}
\newtheorem{lem}[thm]{Lemma}
\newtheorem{prop}[thm]{Proposition}

\newtheorem{rmk}[thm]{Remark}
\newtheorem{definition}[thm]{Definition}

\DeclareMathOperator{\dist}{dist}

\DeclareMathOperator{\cof}{cof}
\DeclareMathOperator{\supp}{supp}
\DeclareMathOperator{\interior}{int}

\newcommand{\RR}{\mathbb{R}}     
\newcommand{\NN}{\mathbb{N}}     

\newcommand{\E}{\mathcal{E}}
  
\newcommand{\J}{\mathcal{J}}  

\newcommand{\e}{\varepsilon}

\renewcommand{\epsilon}{\varepsilon}

\newcommand{\weakstarto}{\overset{\ast}{\rightharpoonup}}

\newcommand{\abs}[1]{\left| #1 \right|}
\newcommand{\norm}[2][]{\left\| #2 \right\|_{#1}}
\newcommand{\inner}[3][]{\left\langle #2,#3\right\rangle_{#1}}

\newcommand{\tang}[2]{T_{#1}(#2)}
\newcommand{\regtang}[2]{\widehat T_{#1}(#2)}
\newcommand{\connormal}[2]{\overline N_{#1}(#2)}

\newcommand{\contact}[1]{C_{#1}}

\newcommand{\Adm}[2]{T_{#1}(\mathcal{E} \ifx&#2& \else \cap #2 \fi)}
\newcommand{\Admr}[2]{T^0_{#1}(\mathcal{E} \ifx&#2& \else \cap #2 \fi)}

\hypersetup{colorlinks = true, linkcolor = blue, citecolor = red}
         
\numberwithin{equation}{section} 

\begin{document}


\title[]{Inertial (self-)collisions of viscoelastic solids with Lipschitz boundaries}
\author[] {Anton\'in \v{C}e\v{s}\'ik}
\address{Department of Mathematical Analysis\\Faculty of Mathematics and Physics\\ 
Charles University\\Prague\\ Czech Republic}
\email {cesik@karlin.mff.cuni.cz}
\author[] {Giovanni Gravina} 
\address{School of Mathematical and Statistical Sciences, Arizona State University, Tempe, Arizona, USA}
\email {ggravina@asu.edu} 
\author[] {Malte Kampschulte}
\address{Department of Mathematical Analysis\\Faculty of Mathematics and Physics\\ 
Charles University\\Prague\\ Czech Republic}
\email {kampschulte@karlin.mff.cuni.cz}

\begin{abstract}
We continue our study, started in \cite{Cesik2022}, of (self-)collisions of viscoelastic solids in an inertial regime. We show existence of weak solutions with a corresponding  contact force measure in the case of solids with only Lipschitz-regular boundaries. This necessitates a careful study of different concepts of tangent and normal cones and the role these play both in the proofs and in the formulation of the problem itself. Consistent with our previous approach, we study contact without resorting to penalization, i.e.\@ by only relying on a strict non-interpenetration condition. Additionally, we improve the strategies of our previous proof, eliminating the need for regularization terms across all levels of approximation.
\end{abstract}

\maketitle


\section{Introduction}

As in our previous paper \cite{Cesik2022}, we consider the (self-)collision of viscoelastic solids under the influence of inertia. That is, we are interested in showing the existence of weak solutions to

\begin{align}
 \label{eq:strongSolution}
 \rho \partial_{tt} \eta + DE(\eta) + D_2R(\eta,\partial_t \eta) = f+\sigma
\end{align}
where $\eta \colon Q \to \RR^n$ is a deformation of a reference configuration, $\rho$ is the mass density, $E$ and $R$ denote elastic energy and dissipation potential, respectively, $f$ is a given force, and $\sigma$ is a solution-dependent contact force. This force purely derives from a non-interpenetration condition, i.e.\@ from the assumption that all deformations are restricted to some container $\Omega$ and are injective almost everywhere.

While also in our previous work we did not prescribe the topology or the general shape of the solids, we restricted ourselves to the case of smooth (i.e.\@ $C^{1,\alpha}$) boundaries. This was a natural class, compatible with the regularity imposed on deformations by the type of second order elastic energies considered there. 

However, the main reason for this restriction was that it not only greatly simplifies the type of contact that can occur, but also gave us a unique, well-defined normal vector at every point of the boundary. Furthermore, this normal vector was not only continuous along the boundary, but also stable with respect to convergence of deformations, two properties we relied on heavily in our proofs.

It is worth mentioning that the study of the static minimization case has a long history, and for this it has long been known how to properly generalize obstacle problems to encompass less regular, i.e.\@ Lipschitz, boundaries (see, for example, \cite{Palmer2018}; see also \cite{Schuricht2002}).

A way to deal with non-smooth boundaries is to work with normals and tangents in a generalized sense. Specifically, variational analysis offers different notions of generalized tangents and normals to any set in $\RR^n$, cf.\@ \cite{Rockafellar1998}. This is the approach that we take and which will be discussed in \Cref{sec:tangent-normal}.
An alternative would be to use the Clarke subdifferential for Lipschitz functions (see \cite{Clarke1987}), and through it define generalized tangents and normals for Lipschitz sets.

The aim of this paper is now to apply those considerations to the dynamic case using the time-delayed approximations developed in \cite{Benesova2020}. This generalization comes with non-trivial analytical challenges. Indeed, not only does it require a more nuanced discussion of contact forces compared to \cite{Cesik2022}, but in addition, a crucial step of the proof (the time-delayed energy inequality), which relied on testing with the time-derivative, can now no longer be obtained in the same way (see Remark \ref{rmk:energyEst}). We instead have to employ a completely different strategy to obtain a similar estimate.

With this, the outline of this paper is as follows: In the next section, we introduce our assumptions and state the main results. Sections \ref{sec:tangent-normal}-\ref{sec:adm} will then be concerned with introducing notation and properties of tangent and normal cones, contact forces, and admissible test functions, respectively. In Section \ref{sec:qs} we will show existence of weak solutions for the quasistatic problem as an intermediate step. Finally, in Section \ref{sec:inertial} we will then use the results from the previous section as a building block for proving the main theorem, thus establishing the existence of weak solutions for the full problem with inertia.

\subsection*{Acknowledgements}
 A.\v{C}.\@ and M.K.\@ acknowledge the support of the the ERC-CZ grant LL2105, and the Czech Science Foundation (GA\v{C}R) under grant No.\@ 23-04766S.  A.\v{C}.\@ further acknowledges the support of Charles University, project GA UK No.\@ 393421.

\section{Assumptions and main results}

\subsection{Viscoelastic solids}

The assumptions we impose on the solid mirror those found in \cite{Cesik2022}, with the key difference that in order to deal with corners and related phenomena, in this paper we expanded our scope to encompass general Lipschitz boundaries, instead of requiring them to be of class $C^{1,\alpha}$. While we will reintroduce the definitions for clarity, we refer the reader to \cite{Cesik2022} for a more detailed discussion of the various aspects that are not directly tied to the lowered regularity of $\partial Q$.

The solid body will be described in Lagrangian coordinates by a (time dependent) deformation of a reference configuration $Q \subset \RR^n$, denoted by $\eta \colon [0, T] \times Q \to \Omega$. The set $Q \subset \RR^n$ will be a Lipschitz, bounded domain, or alternatively a disjoint union of finitely many of such domains. The confining set $\Omega \subset \RR^n$ will similarly be a (possibly unbounded) Lipschitz domain. The class of admissible deformations consists of maps that satisfy the Ciarlet--Ne\v{c}as condition \cite{Ciarlet1987} and are thus injective almost everywhere. To be precise, we set
\begin{equation}
\label{domain-def}
\E \coloneqq \left\{\eta \in W^{2, p}(Q; \RR^n) : \eta(Q) \subset \Omega, \eta|_\Gamma = \eta_0, \det \nabla \eta > 0, \text{ and } \mathcal{L}^n(\eta(Q)) = \int_{Q}\det \nabla \eta(x)\,dx \right\}.
\end{equation}
Here we use $\eta_0$ to denote a given admissible (initial) deformation, $\mathcal{L}^n$ is the Lebesgue measure and $\Gamma$ is a (fixed) measurable subset of $\partial Q$, not necessarily with positive measure, i.e.\ possibly empty. Furthermore, we require $\eta_0|_\Gamma$ to be injective and $\eta_0(\Gamma) \cap \partial \Omega = \emptyset$ (see \cite[Remark 2.3]{Cesik2022} for more information). Here and in the following we assume that $p > n$. Thus we can assume that $\eta \in C^{1, 1 - \frac{n}{p}}(\overline Q;\RR^n)$ for all $\eta\in \mathcal E$ and always work with this precise representative.

In order to simplify the notation, in the following we write $A\langle u \rangle$ as a shorthand for the duality pairing $\langle A, u\rangle_{(W^{k, p})^* \times W^{k, p}}$, whenever $A \colon W^{k, p}(Q; \RR^n) \to \RR$ is linear and $u \in W^{k, p}(Q; \RR^n)$. Moreover, we use the subscript $\Gamma$ to denote spaces of functions whose trace vanishes on that set, e.g.\@ $W^{k,p}_\Gamma(Q) \coloneqq \{u \in W^{k,p}(Q) : u|_\Gamma = 0\}$.

Next, we specify the assumptions on the energy-dissipation pair $(E, R)$. We assume that the elastic energy $E \colon W^{2,p}(Q; \RR^n) \to (-\infty, \infty]$ has the following properties: 
\begin{enumerate}[label=(E.\arabic*), ref=E.\arabic*]
\item \label{E1} There exists $E_{\min} > - \infty$ such that $E(\eta) \ge E_{\min}$ for all $\eta \in W^{2, p}(Q; \RR^n)$. Moreover, $E(\eta) < \infty$ for every $\eta \in W^{2, p}(Q; \RR^n)$ with $\inf_Q\det \nabla \eta > 0$.
\item \label{E2} For every $E_0  \ge E_{\min}$ there exists $\e_0 > 0$ such that $\det \nabla \eta \ge \e_0$ on $Q$ for all $\eta$ with $E(\eta) \le E_0$.
\item \label{E3} For every $E_0 \ge E_{\min}$ there exists a constant $C$ such that 
\[
\|\nabla^2 \eta\|_{L^p} \le C 
\]
for all $\eta$ with $E(\eta) < E_0$.
\item \label{E4} $E$ is weakly lower semicontinuous, that is, 
\[
E(\eta) \le \liminf_{k \to \infty} E(\eta_k)
\]
whenever $\eta_{k} \rightharpoonup \eta$ in $W^{2, p}(Q; \RR^n)$. Moreover, $E$ is continuous with respect to strong convergence in $W^{2, p}(Q; \RR^n)$.
\item \label{E5} $E$ is differentiable in its effective domain with derivative $DE(\eta) \in (W^{2, p}(Q; \RR^n))^*$ given by 
\[
 DE(\eta) \langle \varphi \rangle = \left.\frac{d}{d\e} E(\eta + \e \varphi)\right|_{\mathrlap{\e = 0}}. 
\]
Furthermore, $DE$ is bounded on any sub-level set of $E$ and $DE(\eta_k) \langle \varphi \rangle \to DE(\eta)\langle \varphi \rangle$ whenever $\eta_k \to \eta$ in $W^{2, p}(K; \RR^n)$ for all $K$ compactly contained in $\overline{Q}$ with $\dist(K, \Gamma) > 0$ and $\varphi \in W^{2, p}_{\Gamma}(Q; \RR^n)$.

\item \label{E6} $DE$ satisfies 
\[
\liminf_{k \to \infty}  (DE(\eta_k) - DE(\eta)) \langle (\eta_k - \eta)\psi \rangle \ge 0
\]
for all $\psi \in C^{\infty}_{\Gamma}(Q; [0, 1])$ and all sequences $\eta_k \rightharpoonup \eta$ in $W^{2, p}(Q; \RR^n)$. In addition, $DE$ satisfies the following Minty-type property: If 
\[
\liminf_{k \to \infty} (DE(\eta_k) - DE(\eta))\langle (\eta_k - \eta) \psi \rangle \le 0
\]
for all $\psi \in C^{\infty}_{\Gamma}(Q; [0, 1])$, then $\eta_k \to \eta$ in $W^{2, p}(K; \RR^n)$ for all $K$ compactly contained in $\overline{Q}$ with $\dist(K, \Gamma) > 0$.
\end{enumerate}

Additionally, we assume that the dissipation potential $R \colon W^{2, p}(Q; \RR^n) \times W^{1, 2}(Q; \RR^n) \to [0, \infty)$ is function satisfying the following properties:

\begin{enumerate}[label=(R.\arabic*), ref=R.\arabic*]
\item \label{R1} $R$ is weakly lower semicontinuous in its second argument, that is, for all $\eta \in W^{2, p}(Q; \RR^n)$ and every $b_{k} \rightharpoonup b$ in $W^{1, 2}(Q; \RR^n)$ we have that
\[
R(\eta, b) \le \liminf_{k \to \infty} R(\eta, b_k).
\]
\item \label{R2} $R$ is convex and homogeneous of degree $2$ with respect to its second argument, that is, 
\[
R(\eta, \lambda b) = \lambda^2 R(\eta, b)
\]
for all $\lambda \in \RR$.
\item \label{R3} $R$ admits the following Korn-type inequality: For any $\e_0 > 0$, there exists $K_R$ such that
\[
K_R \|b\|_{W^{1,2}}^2 \le \|b\|_{L^2}^2 +  R(\eta, b)
\]
for all $\eta \in \mathcal{E}$ with $\det \nabla \eta \geq \e_0$ and all $b \in W^{1, 2}(Q; \RR^n)$. 
\item \label{R4} $R$ is differentiable in its second argument, with derivative $D_2R(\eta, b) \in (W^{1, 2}(Q; \RR^n))^*$ given by
\[
 D_2R(\eta, b)\langle \varphi \rangle \coloneqq \left.\frac{d}{d\e} R(\eta, b + \e \varphi)\right|_{\mathrlap{\e = 0}}. 
\]
Furthermore, the map $(\eta, b) \mapsto D_2R(\eta, b)$ is bounded and weakly continuous with respect to both arguments, that is, 
\[
\lim_{k \to \infty}  D_2R(\eta_k, b_k) \langle \varphi \rangle =  D_2R(\eta, b) \langle \varphi \rangle
\]
holds for all $\varphi \in W^{1, 2}(Q; \RR^n)$ and all sequences $\eta_k \rightharpoonup \eta$ in $W^{2, p}(Q; \RR^n)$ and $b_k \rightharpoonup b$ in $W^{1, 2}(Q; \RR^n)$.
\end{enumerate}
We also introduce a variant of \eqref{R3} that will be used for the quasistatic case in the form of
\begin{enumerate}[label=(R.3\textsubscript{q}), ref=R.3\textsubscript{q}]
 \item \label{R3q} $R$ admits the following Korn-type inequality: For any $\e_0 > 0$, there exists $K_R$ such that
\[
K_R \|b\|_{W^{1,2}}^2 \le  R(\eta, b)
\]
for all $\eta \in \mathcal{E}$ with $\det \nabla \eta \geq \e_0$ and all $b \in W^{1,2}_\Gamma(Q;\RR^n)$.
\end{enumerate}

We mention here that the assumptions on the energy-dissipation pair are standard within the framework of second-order viscoelastic materials (see in particular \cite{Healey2009}, \cite{kromerQuasistaticViscoelasticitySelfcontact2019}, and the references therein).

Finally, as in \cite[Remark 2.1]{Cesik2022} we note that these assumptions imply the following:
 \begin{align} \label{eq:R1homog}
  D_2R(\eta,\lambda b) &= \lambda D_2R(\eta,b) \\
  \label{eq:D2Rgrowthbounds}
  \norm[(W^{1,2})^*]{D_2R(\eta,b)} &\leq C \norm[W^{1,2}]{b},   \\ \label{eq:Rgrowthbounds}
  2R(\eta,b) &\leq C \norm[W^{1,2}]{b}^2 \end{align}
 for all $b \in W^{1,2}(Q;\RR^n)$ and all $\eta \in W^{2,p}(Q;\RR^n)$ with $E(\eta) \leq E_0$ and $C = C(E_0)$.
 
%

\subsection{Statement of the main result} \label{main-sec}
The precise definition of (weak) solution to the initial value problem considered in this paper can be formulated as follows. 
\begin{definition}
\label{var-in-def}
Let $T>0$, $\eta_0 \in \E$, $\eta^* \in L^2(Q; \RR^n)$, and $f \in L^2((0, T); L^2(Q; \RR^n))$ be given. We say that 
\[
\eta \in L^{\infty}((0,T); \E) \cap W^{1,2}((0, T); W^{1, 2}(Q; \RR^n))
\]
with $E(\eta)\in L^\infty((0,T))$ is a \emph{weak solution} to \eqref{eq:strongSolution} in $(0, T)$ if $\eta(0) = \eta_0$ and the variational inequality
\begin{multline}
\label{var-ineq-def}
\int_0^T DE(\eta(t))\langle  \varphi(t) \rangle + D_2R(\eta(t), \partial_t \eta(t))\langle  \varphi(t)\rangle\,dt \\ - \rho \langle \eta^*, \varphi(0)\rangle_{L^2} - \int_0^T \rho \langle \partial_t \eta(t), \partial_t\varphi(t) \rangle_{L^2} \,dt \ge \int_0^T \langle f(t), \varphi(t) \rangle_{L^2}\,dt
\end{multline}
holds for all $\varphi \in C([0, T]; \Admr{\eta}{}) \cap C^1_c([0, T); L^2(Q; \RR^n))$. Here the set $\Admr{\eta}{}$ denotes the class of reduced admissible perturbations for the deformation $\eta$; its precise definition is given below in \Cref{lem:strictly-int-admissible}.

Furthermore, we say that such a function $\eta$ is a \emph{weak solution with a contact force $\sigma$} if additionally it satisfies
\begin{multline*}
\int_0^T DE(\eta(t))\langle  \varphi(t) \rangle + D_2R(\eta(t), \partial_t \eta(t))\langle  \varphi(t)\rangle\,dt \\ - \rho \langle \eta^*, \varphi(0)\rangle_{L^2} - \int_0^T \rho \langle \partial_t \eta(t), \partial_t\varphi(t) \rangle_{L^2} \,dt =  \int_{[0,T]\times \partial Q} \varphi(t,x)\cdot d \sigma(t,x) + \int_0^T   \langle f(t), \varphi(t) \rangle_{L^2}\,dt
\end{multline*}
 for all $\varphi \in C([0, T]; W^{2,p}_{\Gamma}(Q;\RR^n)) \cap C^1_c([0, T); L^2(Q; \RR^n))$, where $\sigma \in M([0,T] \times\partial Q; \RR^n)$ is a contact force satisfying the action-reaction principle in the sense of \Cref{cf-def}.
\end{definition} 

Observe that in view of its regularity, $\eta$ belongs to the space $C_w([0,T];W^{2,p}(Q;\RR^n))$. Therefore, we have that $\eta(t) \in W^{2,p}(Q;\RR^n)$ for all $t \in [0,T]$, and in particular the initial condition $\eta(0) = \eta_0$ holds in the classical sense.

With this in hand, we can state the main result of this paper.

\begin{thm}
\label{main-thm-VI}
Let $E$ and $R$ be as in \eqref{E1}--\eqref{E6} and \eqref{R1}--\eqref{R4}, respectively, and let $T> 0$, $\eta_0 \in \E$, $\eta^* \in L^2(Q; \RR^n)$, and $f \in L^2((0,T); L^2(Q; \RR^n))$ be given. Then \eqref{eq:strongSolution} admits a weak solution with a contact force in $(0,T)$ in the sense of \Cref{var-in-def}, where the resulting contact force $\sigma$ has no concentrations in time. Additionally, this solution satisfies the energy inequality
\begin{align*}
 E(\eta(t)) + \frac{\rho}{2} \norm[L^2]{\partial_t \eta(t)}^2 + \int_0^t 2R(\eta(s),\partial_t\eta(s))\, ds \leq E(\eta_0) + \frac{\rho}{2} \norm[L^2]{\eta^*}^2 + \int_0^t \langle f(s),\partial_t \eta(s)\rangle_{L^2}\, ds
\end{align*}
for $\mathcal{L}^1$-a.e.\@ $t \in [0,T]$.
\end{thm} 

As in our previous paper, a similar quasistatic result is also available in the form of \Cref{AuxExistence}.

It is worth noting that, in contrast to \cite{Cesik2022}, we do not provide a characterization of weak solutions in terms of an inequality for test functions in the whole abstract tangent space to the set of configurations. Instead we only consider a somewhat restricted subspace. This is not an oversight, but a crucial feature of the problem. Indeed, in order to work with a set of solutions that is closed under convergence, we have to contend with situations where forces may appear in such a way that prevents the solids from moving or deforming in a direction that otherwise seems admissible. We defer a more detailed discussion of this issue to \Cref{rmk:testSpaces}, as it requires several definitions that will be given in the intervening sections.

\section{Tangent and normal cones}\label{sec:tangent-normal}

We begin by recalling basic definitions and some properties for cones of generalized tangents and normals. For a more detailed treatment we refer to the monograph \cite{Rockafellar1998}, whose notation we also try to follow. However as notation tends to vary between authors, we aim for this section to be mostly self-contained.

Recall that $Q \subset \RR^n$ is open, bounded, and (strongly) Lipschitz, that is, $\partial Q$ is given locally by the graph of a Lipschitz function. 

\begin{definition}[Cone]
A set $C \subset \RR^n$ is called a \emph{cone} if $w \in C$ implies $\lambda w \in C$ for all $\lambda \geq 0$. For any cone $C$, the \emph{polar cone} of $C$ is defined as 
\[
C^* \coloneqq \{w \in \RR^n: w \cdot v \leq 0 \text{ for all } v \in C\}.
\]
\end{definition}
Note that a cone needs to be neither closed nor convex, whereas the polar cone is by definition both closed and convex. In particular $C^{**}$ will be the closed convex hull of $C$.

\begin{definition}[Tangent and normal cones] \label{T&N}
For $Q$ as above and $x \in \overline Q$, we define the following cones:
\begin{itemize} 
\item[$(i)$] \emph{Tangent cone}\footnote{Here $\limsup$ and $\liminf$ denote the set-theoretic versions in a metric space (cf.\@ \cite[Chapter 4]{Rockafellar1998}).}
\[
\tang Q x \coloneqq \left \{ \lim_{i\to\infty} \frac{x_i-x}{\tau_i} : x_i\to x, x_i\in \overline Q, \tau_i\to 0^+ \right\}= \limsup_{\tau\to 0^+} \frac{Q-x}{\tau};
\]
\item[$(ii)$] \emph{Regular tangent cone} 
\[
\regtang Q x \coloneqq \left \{v \in \RR^n : \forall x_i \to x, x_i\in \overline Q , \exists v_i \in \tang{Q}{x_i}, v_i \to v \right \} = \liminf_{y \to x, y\in \overline Q} \tang{Q}{y}; 
\]
\item[$(iii)$] \emph{Convexified normal cone} 
\[
 \connormal Q x \coloneqq \regtang{Q}{x}^*.
\]
\end{itemize}

We also define the respective cones in the deformed configuration via
\[
\tang \eta x \coloneqq \tang {\eta(Q \cap B_{r}(x))}{\eta(x)} 
\]
for any $\eta \in \mathcal{E}$ with $E(\eta) < \infty$ and $r > 0$ such that $\eta$ is injective on $\overline{Q} \cap B_r(x)$.\footnote{Notice that we cannot simply use $\tang{\eta(Q)}{\eta(x)}$. While this would work in the absence of contact, when collisions do occur, we need to distinguish the different, physically unconnected parts of the solid, which this definition fails to do. Also note that this is well defined, as any $\eta$ of finite energy is locally injective by \Cref{contact-set-properties} $(iv)$.} Similar notations are used for the other cones introduced above. We also use the shorthand $\tang{\eta}{t,x} \coloneqq \tang{\eta(t)}{x}$ when dealing with time-dependent deformations. 
\end{definition}

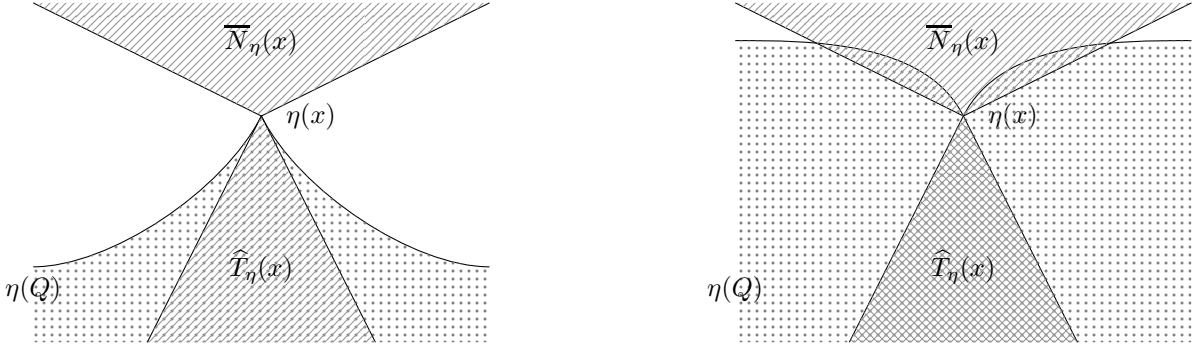
\begin{figure}[ht] 
\centering
\begin{subfigure}{.42\textwidth}
\begin{tikzpicture}
 \node[right] at (0.2,0) {$\eta(x)$};
 \draw (-3,-2) .. controls +(1,0) and +(-.5,-1) .. (0,0) .. controls +(.5,-1) and +(-1,0) .. (3,-2);
 \fill[pattern=dots,pattern color=black!50] (-3,-2) .. controls +(1,0) and +(-.5,-1) .. (0,0) .. controls +(.5,-1) and +(-1,0) .. (3,-2) -- (3,-3) -- (-3,-3);
 \node[below] at (-3,-2) {$\eta(Q)$};
 \draw (-1.5,-3) -- (0,0) -- (1.5,-3);
 \fill[pattern= north east lines,pattern color=black!40]  (-1.5,-3) -- (0,0) -- (1.5,-3) -- (-1.5,-3);
 \node at (0,-2) {$\regtang \eta x$};
 \draw (-3,1.5) -- (0,0) -- (3,1.5);
 \fill[pattern=north east lines,pattern color=black!40] (-3,1.5) -- (0,0) -- (3,1.5) -- cycle;
 \node at (0,1) {$\connormal{\eta}{x}$};
\end{tikzpicture}
\end{subfigure}
\hfill 
\begin{subfigure}{.42\textwidth}
\begin{tikzpicture}
 \node[right] at (0.2,0) {$\eta(x)$};
 \draw (-3,1) .. controls +(1,0) and +(-.5,1) .. (0,0) .. controls +(.5,1) and +(-1,0) .. (3,1);
 \fill[pattern=dots,pattern color=black!50] (-3,1) .. controls +(1,0) and +(-.5,1) .. (0,0) .. controls +(.5,1) and +(-1,0) .. (3,1) -- (3,-3) -- (-3,-3)-- (-3,1);
 \node[below] at (-3,-2) {$\eta(Q)$};
 \draw (-1.5,-3) -- (0,0) -- (1.5,-3);
 \fill[pattern= crosshatch,pattern color=black!40]  (-1.5,-3) -- (0,0) -- (1.5,-3) -- (-1.5,-3);
 \node at (0,-2) {$\regtang \eta x$};
 \draw (-3,1.5) -- (0,0) -- (3,1.5);
 \fill[pattern=north east lines,pattern color=black!40] (-3,1.5) -- (0,0) -- (3,1.5) -- cycle;
 \node at (0,1) {$\connormal{\eta}{x}$};
\end{tikzpicture}
\end{subfigure} 

 \caption{\label{fig:tanCone}Interior regular tangent and exterior convexified normal cone for corners with acute and obtuse angle respectively.}
\end{figure}
%
%
%

Here we recall that while $\tang Q x$ describes the usual set of tangent vectors at $x$, the regular tangent cone $\regtang Q x$ describes the smaller set of vectors which are (up to a vanishing error) also tangent vectors to any point in a neighborhood of $x$. See Figure \ref{fig:tanCone} for an illustration.

\begin{rmk} 
\label{TN-antipol}
The definition of $\connormal{Q}{x}$ we use is not the classic definition. Instead, one usually proceeds via some intermediate steps. First, one constructs the regular normal cone as the polar of $\tang{Q}{x}$. Then, similarly to $\regtang{Q}{x}$, one takes the normal cone to be the corresponding $\limsup$. Finally, since the resulting set can be nonconvex, one takes the convex hull. However, it follows from \cite[Theorem 6.28]{Rockafellar1998} that these two approaches are equivalent. Thus, since we will only ever use $\connormal{Q}{x}$ as a polar cone, we have opted to take this directly as a definition.
\end{rmk}

Notice that the definition of $\regtang Q x$ readily implies that if for $w\in\RR^n$ there are $\epsilon_0>0$ and $r>0$ such that $y + \epsilon w \in \overline{Q}$ for all $\epsilon \in(0,\epsilon_0]$ and all $y\in B_r(x)$, then $w \in \regtang Qx$. More interestingly, in the interior of $\regtang Qx$ the converse holds as well, provided that the condition above holds for a neighborhood of $w$. To be more precise, one can show the following:

\begin{prop}[{\cite[Theorem 6.36]{Rockafellar1998}}] \label{regtang-char}
For $x\in\partial Q$, the following are equivalent:
\begin{itemize}
\item[$(i)$] $w\in \interior \regtang Q x$;
\item[$(ii)$] There exist $\epsilon_0>0$, $r>0$, $\delta>0$ such that $y + \epsilon v \in \overline{Q}$ for all $\epsilon \in [0,\epsilon_0]$, all $y \in \overline{Q} \cap B_r(x)$, and all $v \in B_\delta(w)$.
\end{itemize}
\end{prop}

As we deal with a Lagrangian representation of the solid, it is useful to remember how tangential and normal cones in the image relate to their respective counterparts in the reference configuration. For this we note their transformation behavior.

\begin{lem}[Transformation behavior of cones] For any $\eta \in \E$ and any $x \in \partial Q$ we have
\label{tangent-transform}
\begin{equation*}
\regtang{\eta}{x}= [\nabla\eta(x)]\regtang{Q}{x} \qquad \text{ and } \qquad
\connormal \eta x=[\cof \nabla \eta (x)] \connormal Q x,
\end{equation*}
where $\cof$ denotes the cofactor matrix.
\end{lem}
\begin{proof} For any sequence $\{x_i\}_i \subset \overline Q$ with $(x_i-x)/\tau_i \to v$ where $\tau_i \to 0^+$, we have that 
\[
\frac{\eta(x_i)-\eta(x)}{\tau_i} \to [\nabla \eta(x)] v.
\]
As $\eta$ is locally invertible, the analogous formula holds for $\eta^{-1}$. This implies the first statement. The second statement follows from the definition of $\connormal{Q}{x}$ and the well-known formula $\nabla \eta(x) \cof \nabla \eta (x)^T  = (\det \nabla \eta) \mathrm{I}$.
\end{proof}

Finally, throughout the remaining sections we will frequently use the fact that a Lipschitz boundary implies that the tangent cone cannot be too degenerate. This is made precise in the following lemma.

\begin{lem}\label{lem:lipschitz-unif-tangent}
The following hold:
\begin{itemize}
\item[$(i)$] For every $x \in \partial Q$ we have that $\interior \regtang Q x \neq\emptyset$. In particular, $ \regtang Q x$ contains a cone of the form
\begin{equation}
\label{innercone}
L_{\theta,v} \coloneqq \{w \in \RR^n : w \cdot v \ge \cos \theta |w||v|\}
\end{equation}
(i.e., $L_{\theta, v}$ denotes the cone in the direction $v \in \RR^n \setminus \{0\}$, with opening angle $\theta$), where $\theta$ can be chosen independently of $x$, depending only on the Lipschitz continuity of $\partial Q$, and $v$ can be chosen to attain one of finitely many values $v_1,\dots, v_{k}$ with $|v_i| = 1$.
\item[$(ii)$] Given $E_0 \geq E_{\min}$, there exist $\vartheta > 0$, vectors $v_1, \dots, v_k \in \RR^n$, a covering $\{G_1, \dots, G_k\}$ of $\partial Q$, and points $x_i \in G_i$ such that for every $\eta \in \E$ with $E(\eta)\leq E_0$ we have that for every $x \in G_i$ it holds $L_{\vartheta,w_i} \subset \regtang{\eta}{x}$, where $w_i \coloneqq [\nabla \eta(x_i)]v_i$. Note that $\vartheta$, $k$, and the vectors $v_i$ depend only on $E_0$ and $Q$. In particular, these can be chosen in such a way that they do not depend on $\eta$.
\end{itemize}
\end{lem}

\begin{proof}
We divide the proof into two steps.
\newline
\emph{Step 1:} Recall that by our assumptions on $Q$ there exists a finite covering of $\partial Q$, namely $\{G_1,\dots, G_{k}\}$, such that for every $i = 1, \dots k$ we have that $\Gamma_i \coloneqq \partial Q \cap G_i$ is the graph of a Lipschitz function in the direction of $v_i \in \RR^n$, $|v_i|=1$, and furthermore that $Q \cap G_i$ is the region above the graph. Then for every $i$ we must have that 
\[
y + L_{\theta, v_i} \cap B_{r_y}(y) \subset \overline{Q} \cap G_i
\]
for some $r_y > 0$ which depends on the distance of $y$ to the relative boundary of $\Gamma_i$ in $\partial Q$. In particular, this implies that $L_{\theta,v_i} \subset \tang Q y$ for all $y\in \Gamma_i$, which in turn shows that $L_{\theta,v_i} \subset \regtang Q x$ for all $x \in \Gamma_i$. This concludes the proof of the first statement.
\newline
\emph{Step 2:} The second result follows from similar considerations. To see this, consider the cone $L_{\theta, v}$. Then, using the fact that $\nabla \eta$ is uniformly continuous (with modulus of continuity that depends only on $E_0$) and that $\det \nabla \eta > 0$ in $\overline{Q}$ (see \eqref{E2}), we must have that the transformed cone $[\nabla \eta(x)]L_{\theta,v}$ contains a cone of the form $L_{\vartheta, w}$, where $\vartheta$ depends only on $\theta$ and $E_0$, and $w \coloneqq [\nabla\eta(x)]v$. For $i = 1, \dots, k$, let $G_i$, $x_i$, and $v_i$ be as in the previous step and set $w_i \coloneqq [\nabla \eta(x)]v_i$. Then, for every $x \in G_i$ we have that
\[
L_{\vartheta,w_i} \subset [\nabla \eta(x)]L_{\theta, v_i} \subset [\nabla \eta(x)] \regtang Q x.
\]
The desired result follows from an application of \Cref{tangent-transform}.
\end{proof}

We conclude with a technical lemma.

\begin{lem}
\label{lem:larger-polar}
Let $C \subset \RR^n$ be a cone with nonempty interior and let $v_0 \in \interior C \setminus \{0\}$. Then there exists $\beta > 0$ such that the cone
\[
K_{\beta} \coloneqq \{w \in \RR^n: w \cdot v \le \sin \beta |w||v| \text{ for all } v \in C\}
\]
(i.e., $K_{\beta}$ is the polar cone of $C$ but with opening angle enlarged by $\beta$) has the property that $v_0 \cdot w < 0 $ for all $w \in K_\beta\setminus \{0\}$.
\end{lem}  
\begin{proof}
Since $v_0$ is in the interior of $C$, there exists $\beta_0 > 0$ such that $L_{\beta_0, v_0} \subset C$ (see the definition of $L_{\beta_0,v_0}$ in \eqref{innercone}). We claim that any $0 < \beta < \beta_0$ has the desired property. To prove the claim, let $w \in K_\beta \setminus \{0\}$. Notice that if $w = c v_0$ for $c \in \RR$, then necessarily $c < 0$ and there is nothing else to do. Otherwise, we can find $0 \neq v \in \partial C$ such that $v_0, v$, and $w$ are coplanar and $v$ lies between $v_0$ and $w$. Such a vector $v \in \partial C$ can be found by letting $v \coloneqq av_0 + bw$ for some choice of $a, b > 0$. Then, since $L_{ \beta_0, v_0} \subset C$, we have that the angle between $v_0$ and $v$ must be at least $\beta_0$. Furthermore, using the definition of $K_\beta$ we see that the angle between $v$ and $w$ must be at least $\pi/2 - \beta$. This shows that the angle between $v_0$ and $w$ is strictly larger than $\pi/2$, and therefore concludes the proof.
\end{proof}


\section{Description of the geometry and forces at collision}\label{sec:geometry}
In this section we introduce the tools required in our analysis of both quasistatic and dynamic collisions. In the following we let $I \subset \RR$ be a closed and bounded time interval. 

\begin{definition}[Contact set]
\label{contact-set-def}
For every $\eta \in \E$, the \emph{contact set} of $\eta$ is defined via
\begin{equation}
\label{C-eta-def}
\contact{\eta} \coloneqq \{x \in \overline{Q} : \eta(x) \in \partial \Omega \text{ or } \eta^{-1}(\eta(x)) \neq \{x\}\}.
\end{equation}
For time-dependent deformations, that is, if $\eta \colon I \times \overline Q \to \RR^n$ is such that $\eta(t) \in \E$ for all $t \in I$, we define the \emph{contact set} as
\[
C_\eta \coloneqq \{(t,x) \in I\times \overline{Q}: x \in C_{\eta(t)}\},
\]
where $C_{\eta(t)}$ is the contact set of $\eta(t, \cdot)$ as defined in \eqref{C-eta-def}.
\end{definition}

The following lemma collects useful properties of the contact set.  

\begin{lem}
\label{contact-set-properties}
Let $\eta \in \E$ be given and let $C_{\eta}$ be as in \Cref{contact-set-def}. Then the following hold:
\begin{itemize}
\item[$(i)$] $C_{\eta}$ is a closed subset of $\partial Q$; 
\item[$(ii)$] There exists a number $M\in\NN$ that depends only on $Q$ and $E(\eta)$ such that $\eta^{-1}(\eta(x))$ consists of at most $M$ points;
\item[$(iii)$] If $\eta(x) = \eta(y)$ for some $x \neq y$, then $\interior \regtang \eta x \cap \interior \regtang \eta y = \emptyset$;
\item[$(iv)$] If $\eta(x) = \eta(y)$ for some $x\neq y$, then $\abs{x-y} > r$ for some $r>0$ only depending on $E(\eta)$ but not $\eta$ itself.
\end{itemize}
\end{lem}
\begin{proof}
For a proof of the statements in $(i)$ and $(iii)$ we refer the reader to \cite[Lemma 2.1]{Palmer2018} (note the slightly different notation there); $(ii)$ is a direct consequence of \Cref{lem:lipschitz-unif-tangent}. 
 Finally, $(iii)$ implies that $x$ and $y$ cannot lie in the same set $G_i$ of Lemma \ref{lem:lipschitz-unif-tangent}. Since these sets form a finite open cover that is independent of $\eta$, this implies the existence of a minimal distance $r$ as in the statement of $(iv)$. 
\end{proof}

\begin{rmk} \label{local-inj}
We mention here that the uniform local injectivity in \Cref{contact-set-properties} $(iv)$ actually only depends on the regularity and the lower bound on the Jacobian. Since this will be used later in the paper (see the proof of \Cref{lem:strictly-int-admissible}), we recall this argument for the convenience of the reader. 

Assume that $\|\nabla \eta\|_{C^0} < \infty$ and $\inf_Q \det \nabla \eta > 0$. Then, as one can readily check, there exists a constant $c$, depending only on those two quantities, such that $|\nabla \eta(x) v| \ge c|v|$ for all $x \in \overline{Q}$ and all $v \in \RR^n$. Consequently, for every $x, y \in \overline{Q}$ we have that
\begin{equation}
\label{locinj}
|\eta(x) - \eta(y)| \ge |\nabla \eta(x)(x - y)| - o(|x - y|) \ge c|x - y| - o(|x - y|).
\end{equation}
To conclude, observe that the right-hand side of \eqref{locinj} is bounded away from zero whenever $x$ and $y$ are sufficiently close to each other.
\end{rmk}

Recall that every $\sigma \in M(X; \RR^n)$ (that is, every Radon measure on the compact space $X$ with values in $\RR^n$) admits a polar decomposition of the form $d \sigma = g \,d|\sigma|$ in the sense that
\[
\int_X \varphi \cdot d\sigma = \int_X \varphi \cdot g\,d|\sigma|
\] 
for all $\varphi \in C(X; \RR^n)$, where $|\sigma| \in M^+(X)$ is the total variation of $\sigma$ and $g \in L^1(X, |\sigma|; \RR^n)$ is such that $|g| \le 1$. With this at hand, we proceed to define contact forces as follows. 

\begin{definition}[Contact force]
\label{cf-def}
\begin{enumerate}
\item[$(i)$]
Let $\eta \in \E$. A \emph{contact force} for $\eta$ is a vector-valued measure $\sigma \in M(\partial Q;\RR^n)$ with $\supp \sigma \subset \contact\eta$ and the property that it points in the interior normal direction in the sense that the function $g \colon \partial Q \to \RR^n$ in the polar decomposition $d \sigma = g\,d |\sigma|$ satisfies $-g(x) \in \connormal \eta x$ for $\sigma$-a.a.\@ $x \in \contact\eta$. 
\item[$(ii)$]
Let $\eta\colon I\times \overline Q \to \RR^n$ be such that $\eta(t)\in\E$ for every $t\in I$ and assume that $\eta$ is Borel measurable when considered as a mapping from $I$ to $W^{2,p}(Q;\RR^n)$. Then a \emph{contact force} for $\eta$ is a vector-valued measure $\sigma\in M(I\times \partial Q;\RR^n)$ with $\supp \sigma \subset \contact\eta$ and the property that it points in the normal direction in the sense that the function $g \colon I \times \partial Q \to \RR^n$ in the polar decomposition $d\sigma = g\,d|\sigma|$ satisfies $-g(t, x)\in \connormal {\eta} {t,x}$ for $\sigma$-a.e.\@ $(t,x) \in \contact\eta$.
\item[$(iii)$] Additionally, we say that a contact force $\sigma$ \emph{satisfies the action-reaction principle at self-contact} if 
\begin{equation*}
\int_{\partial Q} (\varphi \circ \eta) \cdot d\sigma = 0
\end{equation*}
for all $\varphi \in C_c(\Omega;\RR^n)$. Similarly, for time-dependent deformations, we say that $\sigma$ \emph{satisfies the action-reaction principle at self-contact} if
\begin{equation*}
\int_{ I \times \partial Q} (\varphi \circ \eta) \cdot d\sigma = 0
\end{equation*}
for all $\varphi \in C_c( I  \times \Omega;\RR^n)$, where, to simplify notation we write $(\varphi \circ \eta)(t,x) \coloneqq \varphi(t,\eta(t, x))$.
\end{enumerate}
\end{definition}

Note that, as is consistent with physics, contact forces always point in interior normal direction. However in contrast to \cite{Cesik2022} where we used interior normals, $\connormal{\eta}{x}$ is always pointing in exterior direction, to remain consistent with standard notations in \cite{Rockafellar1998}. Thus, compared to our previous work, the signs in some of the equations are flipped. Note also that the requirement that $\eta$ is Borel measurable in time (see \Cref{cf-def} $(ii)$) is necessary to ensure the compatibility with the Borel measure $\sigma$, without requiring $\eta$ to be continuous with respect to the variable $t$.

\begin{lem}[Characterization of contact forces]\label{cforce-char} 
The following hold.
\begin{itemize}
\item[$(i)$] For $\eta\in \E$, let $\sigma\in M(\partial Q;\RR^n)$ be such that $\supp \sigma\subset \contact\eta$. Then $\sigma$ is a contact force for $\eta$ if and only if
\[
\int_{\partial Q} \varphi \cdot d\sigma \geq 0
\]
for all $\varphi \in C(\partial Q;\RR^n)$ with $\varphi(x)\in \regtang \eta x$ for $x\in \partial Q$.
\item[$(ii)$] Let $\eta \colon I\times \overline Q \to \RR^n$ be such that $\eta(t)\in\E$ for every $t \in I$ and assume that $\eta$ is Borel measurable when considered as a mapping from $I$ to $W^{2,p}(Q;\RR^n)$. Let $\sigma\in M(I\times \partial Q;\RR^n)$ be such that $\supp \sigma\subset \contact\eta$. Then $\sigma$ is a contact force for $\eta$ if and only if
\begin{equation}
\label{dec-char}
\int_{I\times \partial Q} \varphi \cdot d\sigma \geq 0
\end{equation}
for all $\varphi\colon I\times \partial Q$ satisfying $\varphi(t,\cdot)\in C(\partial Q;\RR^n)$, $t\in I$, Borel measurable in $t$ and bounded in $I\times \partial Q$, with $\varphi(t,x)\in \regtang {\eta}{t,x}$ for $(t,x) \in I\times \partial Q$.
\end{itemize}
\end{lem}

\begin{proof}
We only present the proof of the time-dependent characterization in statement $(ii)$. Indeed, the proof of $(i)$ follows from analogous (but simpler) arguments.
\newline
\emph{Step 1:} Let $\sigma$ be a contact force for $\eta$. Then, by the definition (see \Cref{cf-def} $(ii)$) we have that $d\sigma=g\,d|\sigma|$, where $-g(t,x) \in \connormal \eta {t,x}$ for $\sigma$-a.e.\@ $(t,x) \in C_{\eta}$. Consequently, by the polarity relation between $\regtang \eta {t,x}$ and $\connormal \eta {t,x}$ (i.e., by the definition of $\connormal \eta {t,x}$ given in \Cref{T&N}), we have that $g(t,x) \cdot \varphi(t,x) \geq 0$ $\sigma$-a.e.\@ holds for every $\varphi$ as in the statement. Therefore, 
\[
\int_{I\times\partial Q} \varphi\cdot d\sigma = \int_{I\times \partial Q} \varphi \cdot g \, d|\sigma| \geq 0,
\]
thus proving that \eqref{dec-char} holds as desired. 
\newline
\emph{Step 2:} Next, we show that any measure measure $\sigma$ with $\supp \sigma\subset \contact\eta$ that satisfies \eqref{dec-char} is a contact force for $\eta$. To this end, arguing by contradiction, assume that $\sigma$ is not a contact force for $\eta$. Then we can find a measurable set $S \subset I \times \partial Q$ with $|\sigma|(S)>0$ with the property that $-g(t,x)\notin \connormal \eta {t,x}$ for $(t,x) \in S$. 

Let $h \colon S\to \RR^n$ be Borel measurable with $h(t,x) \in \interior \regtang \eta {t,x}$ and such that $h(t,x) \cdot g(t,x) < 0$ for $\sigma$-a.e.\@ $(t,x) \in S$. Then, an application of Lusin's theorem yields the existence of a compact set $K\subset S$ with $|\sigma|(K)>0$ and a function $\tilde h \in C(I\times \partial Q;\RR^n)$ such that $h(t,x) = \tilde h(t,x)$ for all $(t,x) \in K$. For $t \in I$, we let $K_t \coloneqq \{x\in\partial Q:(t,x)\in K\}$. Moreover, using the fact that $\tilde h$ is uniformly continuous and by \Cref{regtang-char}, for every $t$ we can find a set $G_t \supset K_t$ that is open with respect to the subspace topology of $\partial Q$ and with the property that $\tilde h(t,x) \in \regtang \eta {t,x}$ for all $x \in G_t$. 

Next, consider a sequence of cutoff functions $\psi_i \colon I \times \partial Q \to [0,1]$ that are Borel measurable in time and satisfy $\psi_i(t) \in C(\partial Q;[0,1])$ and $\chi_K \leq \psi_i \leq \chi_G$, and such that $\psi_i \to \chi_K$ in $L^1(I \times \partial Q, \sigma)$ as $i \to \infty$. We can take, for example, $\psi_i(t,x) \coloneqq \max \{1-i\dist(x,K_t),0\}$ for $1/i \leq \dist (K_t,\partial Q \setminus G_t)$. 
Then, denoting $\varphi_i \coloneqq \psi_i \tilde h$ and recalling that $h(t,x) \cdot g(t,x)<0$ on $K$, we obtain that
\[
\int_{I\times \partial Q} \varphi_i \cdot d\sigma = \int_{I\times \partial Q} \psi_i  \tilde h \cdot g \,d|\sigma| \to \int_K h\cdot g\,d|\sigma|<0.
\]
In particular, the same inequality holds for $\varphi \coloneqq \varphi_{n_0}$, for some $i_0 \in \NN$ that is sufficiently large. Notice that $\varphi$ is Borel measurable in time, $\varphi(t) \in C(\partial Q;\RR^n)$, $\varphi(t,x)\in \regtang \eta {t,x}$ for all $(t,x)\in \partial Q$ and that moreover it satisfies
\[
\int_{\partial Q}\varphi \cdot d\sigma < 0.
\]
Thus, we have reached a contradiction to $\eqref{dec-char}$. This concludes the proof.
\end{proof}

%
%

Compactness and closure properties of contact forces will play a fundamental role in \Cref{sec:qs} and \Cref{sec:inertial}, where existence results are obtained via a limiting process involving approximate solutions. 

\begin{thm}[Compactness-closure of contact forces]
\label{compactness-closure-time}
The following hold.
\begin{itemize}
\item[$(i)$] Let $\{\eta_k\}_k \subset \E$ be a sequence of deformation with $E(\eta_k) \leq E_0$ for some $E_0 > E_{\min}$ and assume that there exists $\eta \in \E$ with $E(\eta)\leq E_0$ such that $\eta_k \to \eta$ in $C^1(Q;\RR^n)$. For every $k$, let $\sigma_k$ be a contact force for $\eta_k$ with 
\[
\sup_k \|\sigma_k\|_{M(\partial Q;\RR^n)} < \infty.
\]
Then there exist a subsequence (which we do not relabel) and a limit measure $\sigma$ such that $\sigma_k \weakstarto \sigma$ in $M(\partial Q; \RR^n)$. Moreover, $\sigma$ is a contact force for $\eta$ and if each $\sigma_k$ satisfies the action-reaction principle at self-contact, then so does $\sigma$.
\item[$(ii)$] Let $\{\eta_k\}_k$ be a sequence of time-dependent deformations, that is, for every $k$ we have that $\eta_k \colon I \times \overline{Q} \to \RR^n$ is Borel measurable in time and such that $\eta_k(t) \in \E$ and $E(\eta_k(t)) \le E_0$ hold for all $t \in I$. Furthermore, assume that $\eta_k(t) \to \eta(t)$ in $C^1(\overline Q;\RR^n)$ uniformly in $t$, where $\eta \in C(I;C^{1,\alpha}(\overline Q;\RR^n))$ is such that $\eta(t) \in \E$ and $E(\eta(t)) \le E_0$ for all $t \in I$. For every $k$, let $\sigma_k\in M(I\times\partial Q;\RR^n)$ be a contact force for $\eta_k$ with 
\[
\sup_k\|\sigma_k\|_{M(I\times \partial Q;\RR^n)} < \infty.
\] 
Then there exist a subsequence (which we do not relabel) and a limit measure $\sigma$ such that $\sigma_k \weakstarto \sigma$ in $M(I \times \partial Q;\RR^n)$. Moreover, $\sigma$ is a contact force for $\eta$ and if each $\sigma_k$ satisfies the action-reaction principle at self-contact, then so does $\sigma$.
\end{itemize}
\end{thm}

\begin{proof}
We present the proof of $(ii)$. The proof of the time-independent statement in $(i)$ is analogous but simpler, and therefore we omit it. 

By the weak compactness of measures (eventually extracting a subsequence) we have that $\sigma_k \weakstarto \sigma$. Thus, it remains to verify that $\sigma$ is a contact force for $\eta$.

We begin by showing that $\supp \sigma \subset C_\eta$. We mention here that the proof of this fact follows from the same argument used in \cite[Theorem 3.9]{Cesik2022}. Indeed, for $(t,x)\in \supp \sigma$, using the fact that $\sigma_k \rightharpoonup \sigma$, we have that there are points $(t_k,x_k) \in \supp \sigma_k$ such that $(t_k,x_k) \to (t,x)$. Since $\supp \sigma_k \subset C_{\eta_k}$ by assumption, there are now two cases: Either there is a subsequence such that $\eta_k(t_k,x_k) \in \partial \Omega$, or there exist points $y_k \neq x_k \in \partial Q$ with $\eta_k(t_k,x_k) = \eta_k(t_k,y_k)$. In the first case, the uniform convergence of $\eta_k$ and the uniform continuity of $\eta$ imply that $\eta_k(t_k,x_k) \to \eta(t,x) \in \partial \Omega$ and thus $(t,x) \in \contact\eta$. In the second case, recall that by \Cref{contact-set-properties} $(iv)$ there exists a minimal distance $r > 0$ such that $|x_k - y_k| \geq r$ holds for every $k$. Eventually extracting a further subsequence, we have $y_k \to y \in \partial Q$. As above, the uniform convergence of $\eta_k$ and the uniform continuity of $\eta$ imply that $\eta(t,x) = \eta(t,y)$ with $x \neq y$. Hence, also in this case we have shown that $(t,x)\in \contact{\eta}$. We note here that even though (in contrast to \cite{Cesik2022}) there can be more than two points touching at the same location, this complication is readily circumvented by taking a subsequence of $y_k$.

With this at hand, we can now show that $\sigma$ is indeed a contact force for $\eta$ by using the characterization in \Cref{cforce-char} $(ii)$. To this end, fix $\varphi \colon I\times \partial Q\to \RR^n$ as in \Cref{cforce-char} $(ii)$ and let $\varphi_k$ be defined via
\[
\varphi_k(t,x) \coloneqq \nabla \eta_k(t, x)(\nabla\eta (t,x) )^{-1} \varphi(t,x).
\]
Then, in view of \Cref{tangent-transform} we have that $\varphi_k(t,x) \in \regtang{\eta_k}{t,x}$ and furthermore, by the uniform convergence of $\nabla \eta_k \to \nabla \eta$ in $I \times \partial Q$, we obtain that $\varphi_k \to \varphi$ uniformly in $I \times \partial Q$. Since $\sigma_k$ is a contact force for $\eta_k$, by \Cref{cforce-char} $(ii)$ we have that
\begin{equation*}
\int_{I\times \partial Q} \varphi_k(t,x) \cdot d\sigma_k(t,x) \geq 0.
\end{equation*}
Next, observe that
\begin{equation*}
\int_{I\times \partial Q} \varphi \cdot d\sigma =
 \int_{I\times \partial Q} \varphi \cdot d(\sigma -\sigma_k)+\int_{I\times \partial Q} (\varphi-\varphi_k) \cdot d\sigma_k  + \int_{I\times \partial Q} \varphi_k \cdot d\sigma_k,
\end{equation*}
and that, passing to the limit as $k \to \infty$, the first two terms on the right-hand side go to zero. 
This shows that 
\[
\int_{I\times \partial Q} \varphi \cdot d\sigma = \lim_{k \to \infty} \int_{I \times \partial Q} \varphi_k \cdot d\sigma_k \ge 0,
\] 
and in turn, once again by \Cref{cforce-char} $(ii)$, we obtain that $\sigma$ is a contact force for $\eta$.

Finally, for any $\varphi \in C_c(I\times\Omega;\RR^n)$ we have that $\varphi\circ \eta_k\to \varphi\circ \eta$ uniformly in $I\times \partial Q$. This, together with the convergence $\sigma_k \weakstarto \sigma$, shows that if $\sigma_k$ satisfies the action-reaction principle for every $k$, then the action-reaction principle continues to hold in the limit.
\end{proof}
We conclude this section by noting that $L^2$-in-time estimates remain true in the weak$^*$ limit. 
\begin{lem}\label{cforce-l2intime}
Let $\sigma_k\in L_{w^*}^2( I ;M(\partial Q;\RR^n))$ with $\|\sigma_k\|_{L^2_{w^*}( I ;M(\partial Q;\RR^n))}\leq C$ be such that $\sigma_k \weakstarto \sigma$ in $M(  I \times \partial Q;\RR^n)$. Then $\sigma\in L_{w^*}^2( I ;M(\partial Q;\RR^n))$ and $\|\sigma\|_{L_{w^*}^2( I ;M(\partial Q;\RR^n))}\leq C$.
\end{lem}
For a proof and a discussion of the weak$^*$ measurability involved, we refer to \cite[Remark 3.12, Lemma 3.13]{Cesik2022}.

\section{Admissible directions and test functions}\label{sec:adm}
In this section we study the space of admissible directions. These are, roughly speaking, infinitesimal perturbations of the deformed configuration that do not cause self-interpenetration and for which $\eta(\cdot, \overline{Q}) \subset \overline{\Omega}$. To be precise, we let 
\begin{multline}
\label{AD-def}
T_{\eta}(\E) \coloneqq \{\varphi \in W^{2, p}(Q; \RR^n) : \exists \e_1 > 0 \text{ and } \\ \Phi \in C([0, \e_1); \E) \cap C^1([0, \e_1); W^{2, p}(Q; \RR^n)) \text{ with } \Phi(0) = \eta \text{ and } \Phi'(0^+) = \varphi \}.
\end{multline}
The following result is adapted from \cite[Proposition 3.1]{Palmer2018} and provides a useful characterization of a large subset of $T_{\eta}(\E)$.

\begin{lem}[Strictly interior directions are admissible]
\label{lem:strictly-int-admissible}
For $\eta \in \E$, let
\begin{multline}
\label{SID}
T^0_{\eta}(\E) \coloneqq \{ \varphi \in W^{2, p}(Q;\RR^n) : \varphi|_\Gamma = 0, \varphi(x) \in \interior \regtang\eta x \text{ for all } x \text{ with } \eta(x) \in \partial \Omega, \text{ and } \\ \varphi(x) - \varphi(y) \in \interior \regtang\eta x - \interior \regtang \eta y \text{ for all } x \neq y \text{ with } \eta(x) = \eta(y)\}.
\end{multline} 
Then $\Admr{\eta}{} \subset \Adm{\eta}{}$.
\end{lem}
\begin{proof}
Let $\varphi \in \Admr{\eta}{}$ be given and set $\eta_\epsilon \coloneqq \eta + \epsilon \varphi$. In the following we will show that the function $\Phi(\e) \coloneqq \eta_{\e}$ then satisfies the conditions of \eqref{AD-def}.

First of all, observe that $\eta_{\e} = \eta_0$ when restricted to $\Gamma$. Next, recall that by \eqref{E2} there exists $\e_0 > 0$ that only depends on $E(\eta)$ such that $\det \nabla \eta \geq \epsilon_0$ in $\overline{Q}$. Since $\varphi \in C^{1,\alpha}(Q)$, we have that 
\begin{equation}
\label{lb-etae}
\det \nabla \eta_\epsilon \ge \frac{\epsilon_0}{2},
\end{equation}
provided that $\e$ is sufficiently small. Reasoning as in \Cref{local-inj}, we also obtain that, for all $\e$ such that \eqref{lb-etae} holds, $\eta_\epsilon$ is injective when restricted to balls of radius $r$, where $r$ only depends on $E(\eta)$.

We now claim that $\eta_\epsilon$ is (globally) injective on $\overline Q$ for all $\epsilon >0$ small enough. Indeed, arguing by contradiction, assume that there are a monotone decreasing sequence $\e_i \to 0^+$ and points $x_i \neq y_i \in \overline Q$ such that $\eta_{\epsilon_i}(x_i) = \eta_{\epsilon_i}(y_i)$. By the compactness of $\overline Q$ (up to the extraction of a subsequence, which we do not relabel) we can assume that $x_i \to x$ and $y_i \to y$. We now consider three cases for $\eta(x_i)-\eta(y_i)$ and see that none of them can happen for a subsequence, resulting in a contradiction.

\textit{Case 1: \begin{equation}\eta(y_i)-\eta(x_i) \in \interior \regtang \eta{x}-\interior \regtang \eta{y}.
\end{equation}}
Since the local injectivity of $\eta_{\e_i}$ for $i$ large enough implies that $|x_i - y_i| \ge r$, then necessarily we must have that $|x-y| \ge r$. Using the fact that $\eta_{\e_i} \to \eta$ uniformly we obtain that $\eta(x) = \eta(y)$, and by \Cref{contact-set-properties} $(i)$ we conclude that $x, y \in \partial Q$. Note that if $\eta(y_i) - \eta(x_i) \in \interior \regtang \eta x - \interior\regtang \eta y$ holds for infinitely many values of $i$, then by \Cref{regtang-char} there exists a subsequence such that $\eta(y_i) - \eta(x_i) \in \interior \regtang \eta {x_i} - \interior \regtang\eta{y_i}$ for all $i$ large enough. This means, by \Cref{regtang-char}, that $\eta(B_{r/2}(x_i)\cap Q)$ and $\eta(B_{r/2}(y_i)\cap Q)$ intersect, in contradiction with the interior injectivity of $\eta$, \Cref{contact-set-properties} $(i)$.

\textit{Case 2: \begin{equation}\label{case2} \eta(x_i)=\eta(y_i).\end{equation}} Observe that if \eqref{case2} holds, then we have that $\varphi(x_i) - \varphi(y_i) \in \interior \regtang \eta{x_i}-\interior \regtang \eta{y_i}$ by \eqref{SID}, and therefore $\eta_{\epsilon_i}(x_i) \neq \eta_{\epsilon_i}(y_i)$, as guaranteed by \Cref{contact-set-properties} $(iii)$.

\textit{Case 3: \begin{equation}
\label{proj-not0}
\eta(y_i) - \eta(x_i) \notin (\interior \regtang \eta x - \interior \regtang \eta y) \cup \{0\}.
\end{equation}}
Let $K_{\beta}$ be the enlarged polar cone given by \Cref{lem:larger-polar} with $C = \regtang \eta x - \regtang \eta y$ and $v_0 = \varphi(x) - \varphi(y) \in \interior \regtang \eta x - \interior \regtang \eta y$. Then, for all $w \in K_{\beta}$ we have that
\begin{equation}
\label{enlarged-N}
(\varphi(x) - \varphi(y)) \cdot w < 0.
\end{equation}
Let $p_i$ be the projection of $\eta(y_i) - \eta(x_i)$ onto $K_{\beta}$ and observe $p_i \neq 0$ by \eqref{proj-not0}. Then, using the fact that $p_i \cdot (\eta(y_i) - \eta(x_i)) \ge 0$, we obtain that
\begin{equation}
\label{dot-ei}
\frac{p_i}{|p_i|}\cdot (\eta_{\epsilon_i}(y_i) - \eta_{\epsilon_i}(x_i)) \ge \epsilon_i \frac{p_i}{|p_i|} \cdot (\varphi(y_i)-\varphi(x_i)).
\end{equation}
Notice that (eventually extracting a subsequence) $p_i/|p_i| \to p \in K_{\beta} \setminus \{0\}$, and therefore, by \eqref{enlarged-N} we have that
\begin{equation}
\label{dot>0limit}
\frac{p_i}{|p_i|} \cdot (\varphi(y_i)-\varphi(x_i)) \to p \cdot (\varphi(y)-\varphi(x)) > 0.
\end{equation}
In particular, combining \eqref{dot-ei} and \eqref{dot>0limit}, we obtain that 
\begin{equation}
\label{dot>0i}
 \frac{p_i}{|p_i|}\cdot (\eta_{\epsilon_i}(y_i) - \eta_{\epsilon_i}(x_i)) > 0
\end{equation}
holds for all $i$ sufficiently large. Finally, \eqref{dot>0i} implies that $\eta_{\epsilon_i}(x_i) \neq \eta_{\epsilon_i}(y_i)$.

Therefore we have reached the desired contradiction and thus shown that $\eta_{\epsilon}$ is injective on $\overline Q$ for all $\e$ sufficiently small.

A similar but simpler argument can be used to show that $\eta_{\epsilon}(\overline Q) \subset \Omega$. Indeed, the obstacle $\RR^n \setminus \Omega$ can be treated as a non-moving part of the solid. This concludes the proof.
\end{proof}

\begin{figure}[ht] 
\centering
\begin{subfigure}{.4\textwidth}
\begin{tikzpicture}
 \path (-3,2)--(3,-2);
 \draw (-3,1) -- (0,0) -- (-3,-1);
 \fill[pattern=north east lines, pattern color=black!40] (-3,1) -- (0,0) -- (-3,-1);
 \draw (3,1) -- (0,0) -- (3,-1);
 \fill[pattern=north west lines, pattern color=black!40] (3,1) -- (0,0) -- (3,-1);
 \node[above] at (-0.1,0) {$\eta(x)$};
 \node[below] at (0.1,0) {$\eta(y)$};
 \node at (-2,0) {$\eta(Q)$};
\end{tikzpicture}
\end{subfigure}
\hfill 
\begin{subfigure}{.4\textwidth}
\begin{tikzpicture}
 \draw (-3,1) -- (0,0) -- (-3,-1);
 \fill[pattern=north east lines, pattern color=black!40] (-3,1) -- (0,0) -- (-3,-1) --(-3,-2)-- (3,-2)--(3,2) -- (-3,2) -- (-3,1);
 \draw[dashed] (3,1) -- (0,0) -- (3,-1);
 \fill[pattern=north west lines, pattern color=black!40] (3,1) -- (0,0) -- (3,-1);
 \node[above] at (0,1) {$\varphi \in T_\eta(\E)$};
 \node at (2,0) {$\varphi \in T_\eta^0(\E)$};
\end{tikzpicture}
\end{subfigure}
 \caption{
\label{fig:cornersTang} Comparison of $T_\eta^0(\E)$ and $T_\eta(\E)$ as described in \Cref{rem:testCones}. The first image shows two parts of the solid touching at their corners $x,y$ with $\eta(x)=\eta(y)$. For $y$ as preimage of the right side, the second image illustrates the possible values of $\varphi(y)$.}
\end{figure}
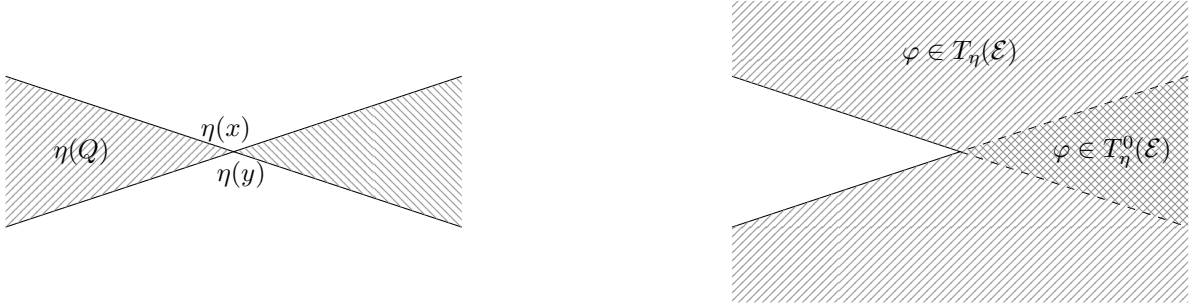

\begin{rmk} \label{rem:testCones}
Note that, in general, the inclusion in \Cref{lem:strictly-int-admissible} is strict. To see this, consider the case of a head-on collision where the contact happens between two outside corners of otherwise smooth, well-separated surfaces. Then the cones $\regtang{\eta}{x}$ and $\regtang{\eta}{y}$ have the same opening angles and the resulting set $T_{\eta}^0(\E)$ is relatively small when compared to $T_{\eta}(\E)$, that is, the set of all the directions it is possible to move in.
 
 In contrast, if we replace one of the these angles with the complement of the other one, creating an inside corner, then at the corner the regular tangent cones will in fact not change, but now the closure of $T_{\eta}^0(\E)$ is precisely the set of directions it is possible to move in. 
 
 In fact, this shows that it is not possible to give a pointwise characterization of $\Adm{\eta}{}$ in terms of the regular tangent cone. It is reasonable to ask if there is a better characterization in terms of e.g.\@ a different definition of tangent cones but to the best of our knowledge there seems to be no way to do so, as the set of admissible directions ultimately depends on the complete behavior (e.g.\@ oscillations and the directions they occur in) of the boundary in a neighborhood of the contact point.
 
 For our purposes however, neither is such characterization actually required, as the preceding lemma gives us a sufficient amount of admissible directions for all required estimates, nor would having such a characterization be particularly helpful. Indeed, while it would allow for a stronger control on the directions of the Lagrange multipliers in the initial minimization, these are not the contact forces occurring in the final equation, which arise only as limits, whose behavior is better captured using regular normal and tangent cones. 
 
 Additionally, since it is not convex, characterizing the set of admissible directions actually does not characterize the set of admissible test functions for the equation. Consider for example the above case of two angles meeting. These can ``slide'' along each other in two directions, so both of these directions provide admissible test functions. But then so does their average, even though moving in that average direction immediately results in an overlap.
 
 All of this is in contrast to the case of a smooth boundary (cmp.\@ \cite{Cesik2022}), where there is no need to distinguish different cones, as all different types of tangential cones for a given point will be the same half-space and there is only ever a single normal direction.
\end{rmk}

To allow us to better quantify contact forces in the proof we now require a specific type of test function.

\begin{definition}
\label{unif-int-def}
Given $\eta\in \E$, we say that $\tilde t_\eta \colon \overline Q \to \RR^n$ is a \emph{uniformly interior vector field} for $\eta$, if there is an angle $\vartheta \in (0, \pi/2)$ and a constant $0<c\leq 1$ such that
\[
\tilde t_\eta(x) \cdot n \leq -\sin \vartheta |\tilde t_\eta(x)||n| \quad \text{and }\quad  c\leq|\tilde t_\eta(x)| =1 \qquad \text{for all} \quad  x\in \partial Q  \text{ and } n \in \connormal \eta x,
\]
and $|\tilde t| $ bounded on $\overline Q$. Similarly, given $\eta\in L^\infty( I; \mathcal E)\cap W^{1,2}( I ;W^{1,2}(Q;\RR^n))$ with $E(\eta)\in L^\infty( I )$, we say that $\tilde t_\eta \colon  I \times \overline Q\to \RR^n$ is a \emph{uniformly interior vector field} for $\eta$, if there is an angle $\vartheta \in (0, \pi/2)$ such that for all $t\in I$
\[
\tilde t_\eta(t,x) \cdot n \leq -\sin \vartheta |\tilde t_\eta(t,x)| |n| \quad \text{and }\quad  |\tilde t_\eta(t,x)| =1 \qquad \text{for all} \quad  x\in \partial Q  \text{ and } n \in \connormal \eta x,
\]
and $|\tilde t(t,\cdot)|$ bounded on $\overline{Q}$.
\end{definition}

Note that the polarity of $\regtang \eta x$ and $\connormal \eta x$ and the strict positivity of $\sin \vartheta$ imply that $\tilde{t}_{\eta}(x)\in \interior\regtang \eta x$. To be precise, the definition implies that the cone $L_{\vartheta, \tilde{t}_{\eta}(x)} \subset \regtang \eta x$ (see \eqref{innercone}). We show below that for any deformation $\eta$ there exists a \emph{smooth} uniformly interior vector field. 

\begin{prop}
\label{smooth-unif-int}
The following hold.
\begin{itemize}
\item[$(i)$] For every $\eta \in \E$ there exists a uniformly interior vector field $\tilde{t}_{\eta}$ in the sense of \Cref{unif-int-def} with $\tilde t_\eta \in C^{k_0}(\overline Q;\RR^n)$ and such that $\|\tilde t_\eta\|_{C^{k_0}(\overline Q;\RR^n)}\leq C$ for all $k_0\in \NN$, where $C$ is a constant that depends only on $E(\eta)$ and $k_0$.
\item[$(ii)$] For every $\eta \in L^\infty(I;\E)\cap W^{1,2}(I;W^{1,2}(Q;\RR^n))$ with $\|E(\eta)\|_{L^\infty(I)} \leq E_0$ there exists a uniformly interior vector field $\tilde t_\eta$ in the sense of \Cref{unif-int-def} with $\tilde t_\eta \in C(I;C^{k_0}(\overline Q;\RR^n))$ and such that $\|\tilde t_\eta\|_{C(I;C^{k_0}(\overline Q;\RR^n)} \leq C_{k_0}$ for all $k_0\in\NN$, where $C_{k_0}$ is a constant that depends only on $E_0$ and $k_0$. Moreover it can be chosen so that $\tilde t_\eta \in W^{1,2}(I;L^2(Q;\RR^n))$ with the estimate $\|\partial_t \tilde t_\eta\|_{L^2(I\times Q;\RR^n)}\leq C\| \partial_t\nabla \eta \|_{L^2 (I\times Q;\RR^n)}$ with $C$ depending on $E_0$.
\end{itemize}
\end{prop}

\begin{proof}
We only present the proof of $(ii)$. The proof of the time-independent statement in $(i)$ follows from a similar argument (see also \cite[Proposition 3.1]{Palmer2018}).
Let $\{G_1,\dots, G_k\}$ be a covering of $\partial Q$ with corresponding points $x_i \in G_i$ and directions $v_i \in S^{n - 1}$ and $\vartheta$ for $E_0$ given as in the proof of \Cref{lem:lipschitz-unif-tangent}. Choose $\delta>0$ so small that $\{G_1,\dots, G_k\}$ covers also $P_\delta$, the $\delta$-neighborhood of $\partial Q$. Note that this $\delta$ still depends only on $Q$ and $E_0$. 
Now let $\{\psi_1,\dots,\psi_k\}$ be a partition of unity on $P_\delta$ subordinated to the covering $\{G_1,\dots,G_k\}$. To be precise, $\psi_i \in C^\infty_c(G_i,[0,1])$ for every $i = 1, \dots, k$ and $\sum_{i=1}^k \psi_i(x)=1$ for all $x \in P_\delta$.

Next, we construct $\tilde t_Q \in C^\infty(Q_\delta;\RR^n)$, where $Q_\delta$ is the $\delta$-neighborhood of $Q$, by setting (we consider each $\psi_i$ to be extended outside $G_i$ by $0$)
\begin{equation*}
\tilde t_Q (x)\coloneqq \sum_{i=1}^k \psi_i (x) v_i,\quad x\in Q_\delta.
\end{equation*}
Then $\tilde t_Q$ is a smooth uniformly interior vector field to $Q$ with angle $\vartheta$, moreover satisfying $ c\leq |\tilde t_Q|\leq 1$ on $P_\delta$, where $0<c\leq 1$ depends only on $Q$.

We consider an extension of $\eta$ to $I\times Q_\delta$ preserving all the norms, in particular the $L^\infty(I;W^{2,p}(Q_\delta;\RR^n))$ and $W^{1,2}(I;W^{1,2}(Q_\delta;\RR^n))$ norms, up to a constant.  With this define $$t_\eta(t,x) = \frac{\nabla\eta(t,x)  \,\tilde t_Q(x) }{|\nabla\eta(t,x)  \,\tilde t_Q(x)|} |\tilde t_Q(x)|, \quad t\in I, \quad x\in Q_\delta.$$ 
Because of the uniform lower bound on the Jacobian which can be extended to $Q_\delta$ for $\delta$ small enough, we see that $ t_\eta \in C(I;C^{1,\alpha}(Q_\delta;\RR^n))$ with H\"older seminorm dependent only on $E_0$. 
Further, we see by an application of the chain rule, that 
$$\|\partial_t  t_\eta (t)\|_{L^2(Q_\delta;\RR^n)}\leq C\|\partial_t \nabla \eta \|_{L^2(Q;\RR^n)} $$ 
with $C$ depending on $E_0$.

Now choose $\tilde \delta>0$, possibly smaller than $\delta$ but only dependent on $E_0$, such that it holds  for all $t\in I$ and all $x\in \partial Q$ and $\tilde x\in Q_\delta$ with $|x-\tilde x|\leq \tilde \delta$ that $$t_\eta(t,x)\cdot t_\eta(t,\tilde x)\geq \cos (\vartheta/2).$$ 

We mollify in space, that is put $\tilde t_\eta \coloneqq t_\eta*\xi_\delta$, where $*$ is the convolution in space and $\xi_\delta$ the smooth mollification kernel with radius $\delta$. Because of the above choice of $\tilde \delta$, one can readily check that $\tilde t_\eta$ is a uniformly interior field for $\eta$ in the sense of \Cref{unif-int-def}, with the angle $\vartheta/2$.
Finally, regarding the regularity of $\tilde t_\eta$, we see that for each $t\in I$ we have
$$\|\tilde t_\eta(t)\|_{C^{k_0}(\overline Q;\RR^n)}\leq C_{k_0}\|t_\eta(t)\|_{C(\partial Q;\RR^n)}$$
with $C_{k_0}$ depending on $E_0$ and $k_0$, and
$$ \|\partial_t \tilde t_\eta (t)\|_{L^2(Q;\RR^n)}\leq  C \| \partial_t t_\eta (t)\|_{L^2(Q;\RR^n)}$$
with $C$ depending on $E_0$, which combining with the above inequalities finishes the proof.
 \end{proof}

Next, we show that the set of strictly interior directions, namely $\Admr{\eta}{}$, is well-behaved with respect to sequences of approximating deformations.

\begin{prop} 
\label{prop:TErecoveryLocal}
The following hold.
\begin{itemize}
\item[$(i)$] Let $\{\eta_k\}_k \subset \E$ be given and assume that there exists $\eta \in \E$ such that $\eta_k \to \eta$ in $C^1(\overline{Q}; \RR^n)$. Then, for every $\varphi \in \Admr{\eta}{}$ there exists a sequence $\{\varphi_k\}_k$ such that $\varphi_k \to \varphi$ in $W^{2, p}(Q; \RR^n)$ and with the property that $\varphi_k \in \Admr{\eta_k}{}$ for all $k$ sufficiently large.
\item[$(ii)$] Let $\{\eta_k\}_k \subset L^{\infty}(I; \E)$ be given and assume that there exists $\eta \in W^{1, 2}(I; W^{2, p}(Q; \RR^n))$ with $\|E(\eta)\|_{L^{\infty}(I)} \le E_0$ such that $\eta_k \to \eta$ and $\nabla \eta_k \to \nabla \eta$ uniformly on $I \times \overline{Q}$. Then, for every $\varphi \in C(I; \Admr{\eta}{})$ there exists a sequence $\{\varphi_k\}_k$ such that $\varphi_k \to \varphi$ in $C(I; W^{2, p}(Q; \RR^n))$ and with the property that $\varphi_k \in C(I; \Admr{\eta_k}{})$ for all $k$ sufficiently large. Moreover, let $J \subset I$ and assume that $\varphi \in C(I; \Admr{\eta}{}) \cap C^1_c(J; L^2(Q; \RR^n))$. Then there exists a sequence $\{\varphi_k\}_k$ as above but with the additional property that $\varphi_k \in C(I; \Admr{\eta_k}{}) \cap C^1_c(J; L^2(Q; \RR^n))$ for all $k$ sufficiently large.
\end{itemize}
\end{prop}

\begin{proof} We only present the proof of the slightly more complicated time-dependent statement $(ii)$. 

Let $\xi_{\Gamma} \colon \overline{Q} \to \RR$ be a smooth function that satisfies $\xi_{\Gamma}(x) = 0$ for $x \in \Gamma$ and $\xi_{\Gamma}(x) > 0$ for $x \in \overline{Q} \setminus \overline{\Gamma}$. Let $\varphi$ be as in $(ii)$ and $\tilde{t}_{\eta}$ be as in \Cref{smooth-unif-int} $(ii)$ (with $k_0 \ge 2$). For $m \in \NN$, we then define
\begin{equation}
\label{unif-int-test}
\varphi_m \coloneqq \varphi + \frac{1}{m}\tilde{t}_{\eta} \xi_{\Gamma}.
\end{equation}
As one can readily check (see \eqref{SID}), we have that $\varphi_m \in C(I; \Admr{\eta}{})$ for all $m$; furthermore, it is evident that $\varphi_m \to \varphi$ in $C(I; W^{2, p}(Q; \RR^n))$ as $m \to \infty$. 

Fix $m > 0$. We claim that $\varphi_m \in C(I; \Admr{\eta_k}{})$ for all $k$ sufficiently large. Indeed, if this is not the case then we can find a subsequence of $\{\eta_k\}_k$ (which we do not relabel) and contact points $\{(t_k, x_k)\}_k$ with $(t_k, x_k) \in C_{\eta_k}$ such that either there exists $\{y_k\}_k$ with $x_k \neq y_k$ and $\eta_k(t_k, x_k) = \eta_k(t_k, y_k)$ and $\varphi_m(t_k, x_k) - \varphi_m(t_k, y_k) \notin \interior \regtang {\eta_k} {t_k, x_k} - \interior \regtang {\eta_k} {t_k, y_k}$ or $\eta_k(t_k, x_k) \in \partial \Omega$ and $\varphi_m(t_k, x_k) \notin \interior \regtang {\eta_k}{t_k, x_k}$.

Since that the approach for handling the latter case is comparable, we consider only the first case. Eventually extracting a further subsequence (which again we do not relabel), we can find $t \in I$ and $x, y \in \partial Q$ such that $t_k \to t$, $x_k \to y$, $y_k \to y$, $x \neq y$, and $\eta(t, x) = \eta(t, y)$.

Recalling that $\varphi_m \in C(I; \Admr{\eta}{})$, by \eqref{SID} and \Cref{smooth-unif-int} $(ii)$, we can find two sets $K_1$ and $K_2$ such that 
\[
\varphi_m(t, x) - \varphi_m(t, y) \in K_1 \subset \subset K_2 \subset \subset \interior \regtang\eta {t, x} - \interior \regtang \eta {t, y}.
\]
In particular, \Cref{regtang-char}, the regularity of $\eta$, and \Cref{tangent-transform} imply that 
\begin{equation}	
\label{unif-tangent-inclusion}
\varphi_m(t', x') - \varphi_m(t', y') \in K_2 \subset \subset \interior \regtang\eta {t', x'} - \interior \regtang \eta {t', y'}
\end{equation}
for every $t'$ near $t$, every $x'$ near $x$ and every $y'$ near $y$. Since by assumption we have that $\nabla \eta_k \to \nabla \eta$ uniformly on $I \times \overline{Q}$, another application of \Cref{tangent-transform} yields that 
\[
K_2 \subset \interior \regtang {\eta_k}{t_k, x_k} - \interior \regtang {\eta_k} {t_k, y_k}
\]
holds for all $k$ sufficiently large, where $K_2$ is given as in \eqref{unif-tangent-inclusion}. 

In turn, this implies that $\varphi_m(t_k, x_k) - \varphi_m(t_k, y_k) \in \interior \regtang {\eta_k}{t_k, x_k} - \interior \regtang {\eta_k} {t_k, y_k}$, and we have therefore reached a contradiction.

To prove the second part of the statement, let 
\[
\varphi_m \coloneqq \varphi + \frac{1}{m}\tilde{t}_{\eta} \xi_{\Gamma}\psi_J,
\]
where $\psi_J$ is an opportunely defined smooth cut-off function to ensure that $\varphi_m \in C^1_c(J; L^2(Q; \RR^n))$. The argument above can then be repeated with only straightforward changes. This concludes the proof. 
\end{proof}

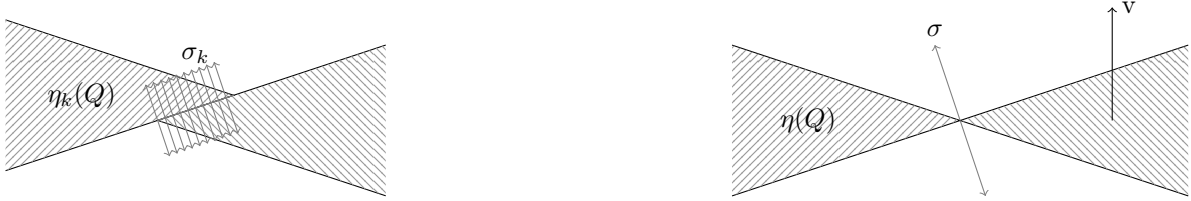
\begin{figure}[ht] 
\centering
\begin{subfigure}{.4\textwidth}
\begin{tikzpicture}
 \draw (-3,1) -- (0,0) -- (-3,-1);
 \fill[pattern=north east lines, pattern color=black!40] (-3,1) -- (0,0) -- (-3,-1);
 \draw (2,.666) -- (-1,-.333) -- (2,-1.333);
 \fill[pattern=north west lines, pattern color=black!40] (2,.666) -- (-1,-.333) -- (2,-1.333);
 \draw[postaction=decorate,decoration={markings,mark=between positions 0 and 1 step 0.1 with { \draw[->] (0,0) -- (0,.5);}},color=gray] (-1,-.333) -- (0,0);
 \draw[postaction=decorate,decoration={markings,mark=between positions 0 and 1 step 0.1 with { \draw[->] (0,0) -- (0,-.5);}},color=gray] (-1,-.333) -- (0,0);
 \node[above]  at (-.5,.3) {$\sigma_k$};
 \node at (-2,0) {$\eta_k(Q)$};
\end{tikzpicture}
\end{subfigure}
\hfill 
\begin{subfigure}{.4\textwidth}
\begin{tikzpicture}
 \draw (-3,1) -- (0,0) -- (-3,-1);
 \fill[pattern=north east lines, pattern color=black!40] (-3,1) -- (0,0) -- (-3,-1);
 \draw (3,1) -- (0,0) -- (3,-1);
 \fill[pattern=north west lines, pattern color=black!40] (3,1) -- (0,0) -- (3,-1);
 \draw[->,color=gray] (0,0) -- (-.333,1);
 \draw[->,color=gray] (0,0) -- (.333,-1);
 \draw[->] (2,0) -- (2,1.5);
 \node[right] at (2,1.5) {v};
 \node[above] at (-.333,1) {$\sigma$};
 \node at (-2,0) {$\eta(Q)$};
\end{tikzpicture}
\end{subfigure}
 \caption{
\label{fig:corners} Approximate configurations and construction of a contact force without a fixed sign for all potential movement directions, as described in Remark \ref{rmk:testSpaces}. Specifically, moving the right hand solid in direction $v$ is part of $\Adm{\eta}{}$ but not $\Admr{\eta}{}$.}
\end{figure}

\begin{rmk}\label{rmk:testSpaces}
\Cref{prop:TErecoveryLocal} already gives us a hint as to why we need to consider the set $\Admr{\eta}{}$ instead of $\Adm{\eta}{}$ in the weak inequality \eqref{var-in-def}. As we construct solutions through approximations, we require some notion of stability under convergence. To be precise, as the set of admissible test functions necessarily depends on the deformation itself, when studying the convergence of approximate solutions, we can only consider test functions that can be approximated by other test functions which are admissible along the sequence of converging deformations.
 
 The sets $\Adm{\eta}{}$ do not have this property. Or spoken in a more abstract way, $\{(\eta,b) : \eta \in \E, b \in \Adm{\eta}{}\}$ is not closed in any reasonable topology. To see this, consider a limit configuration consisting of two corners touching at their tips. This can arise as the limit of a sequence of configurations where the two parts of the solid touch along their sides (see Figure \ref{fig:corners}). Even restricting to rigid motions of one of the sides, it is clear that in the approximation there is a whole half-space of directions that cause overlap and are thus not admissible. In contrast, for the limit configuration, the only rigid motions in $\Adm{\eta}{}$ that are excluded are those that directly cause the tip to enter into the other corner.\footnote{Note also that the resulting set $\Adm{\eta}{}$ is non-convex. As the weak inequality is linear in its test function, this here this would allow to enlarge the set of test functions even further.}
 
 The same issue also directly translates into a consideration about contact forces: As the approximations have well-defined normals, we can have a clearly defined contact force for each of them. Assuming the right scaling, this force can persist in the limit. Indeed, this is precisely the reason why we can only assume that the final contact force has a direction in $\connormal{Q}{x}$ instead of the smaller set $\tang{Q}{x}^*$.
 
 We note here that this does not mean that the set of test functions, nor the set of possible directions for contact forces we choose are optimal. However, this proved to be good enough to obtain a satisfactory existence theory and is likely the best that can be done using local characterizations. For example, for this type of corners there are only two possible directions of contact forces, which are in fact determined by the side of the overlap the approximation is coming from. Even for isolated corners, this argument is difficult to formalize, and still much harder to generalize to an arbitrary setting of Lipschitz boundaries. In contrast, the approach we use is mostly based on a well-established theory and relies mainly on the abstract minimization structure instead of a precise characterization of contact forces or admissible testing directions. It is thus also much easier to generalize to other settings.
\end{rmk}

\begin{rmk}
 Formally, $\Adm{\eta}{}$ is the tangent cone to $\E$ in the space $W^{2,p}(Q;\RR^n)$. One can then ask what the corresponding regular tangent cone is, i.e.\@ $\widehat{T}_\eta(\E) \coloneqq \liminf_{\bar{\eta} \to \eta, \bar{\eta}\in\E} \Adm{\bar{\eta}}{}$. The space $\Admr{\eta}{}$ appears to be a reasonable candidate, as \Cref{prop:TErecoveryLocal} implies that $\Admr{\eta}{} \subset \widehat{T}_\eta(\E)$ and shows that $\Admr{\eta}{}$ behaves similarly with respect to convergence. Conveniently, this would also formally identify the set of contact forces with the corresponding convexified normal cone.
 
 We conjecture that this is indeed a characterization, i.e.\@ $\Admr{\eta}{} = \widehat{T}_\eta(\E)$. It is not hard to convince oneself that it is true in standard cases such as a smooth boundary, where both equal $\Adm{\eta}{}$, or for known geometries by a direct contradiction: If without loss of generality $\varphi \in \Adm{\eta}{}$ with $\varphi(x) \notin \regtang{\eta}{x}$ at contact and the other side is not moving, then there is an approximation $\{x_k\}_k$ of $x$ for which $\varphi(x) \notin \tang{\eta}{x_k}$. Constructing an approximation $(\eta_k)_k$ of $\eta$ with contact at $x_k$ that results in collision when moving in the direction $\varphi(x_k) \approx \varphi(x)$ is then a simple exercise in the case of e.g.\@ isolated corners.
 
 For general Lipschitz boundaries, this proof strategy runs into technical issues. Nevertheless we were not able to find a counterexample and thus leave this characterization as an open problem. In fact our technique is compatible with using $\widehat{T}_\eta(\E)$ in place of $\Admr{\eta}{}$, as we mainly rely its behavior under convergence. However we believe that the explicit characterization that $\Admr{\eta}{}$ offers to be more valuable in practice.
\end{rmk}

\section{Quasistatic evolution of viscoelastic solids with Lipschitz boundaries}\label{sec:qs}
In this section, we study the quasistatic counterpart of our problem. Besides being of independent interest, this will also serve as a building block for the main result of this paper. We mainly follow the strategy presented in \cite[Section 4]{Cesik2022}, in combination with the results of the previous sections, in order to treat the case of Lipschitz boundaries. Thus, we will only sketch the parts of the proof that are identical to those in \cite{Cesik2022} and focus more on the differences. In particular, a crucial step in our analysis involves establishing an energy inequality. For this, we follow a different approach than the one in \cite{Cesik2022}, which allows us to work without an additional regularization (see \Cref{rmk:energyEst} below). This has the advantage that, in contrast to \cite[Section 4]{Cesik2022}, we do not need an additional set of assumptions for the regularized problem.

The quasistatic problem in question is
\begin{equation}
\label{aux-Pb}
DE (\eta) + D_2 R(\eta, \partial_t \eta) = f.
\end{equation}
As a particular case, we obtain an existence theory for the parabolic equation
\begin{equation}
\label{aux-par}
\rho\frac{\partial_t \eta}{h} + DE(\eta) + D_2R(\eta, \partial_t \eta) = f + \rho\frac{\zeta}{h},
\end{equation}
where $\rho\frac{\partial_t \eta}{h}$ can be thought of as derivative of an additional dissipation and $\zeta$ is a given function. Later in \Cref{sec:inertial} we will make the choice $\zeta(t) = \partial_t \eta(t - h)$. Observe that combining these two newly introduced terms yields the difference quotient 
\[
\rho\frac{\partial_t \eta (t) -\partial_t \eta(t-h)}{h}.
\]
Weak solutions to \eqref{aux-Pb} are defined as follows.
\begin{definition}
\label{sol-def-aux}
Let $h > 0$, $\eta_0 \in \E$, and $f \in L^2((0, h); L^2(Q; \RR^n))$ be given. We say that 
\[
\eta \in W^{1,2}((0, h); W^{1,2}(Q; \RR^n)) \cap L^{\infty}((0, h); \E ) \quad \text{with} \quad E(\eta)\in L^\infty((0, h))
\]
is a solution to \eqref{aux-Pb} in $(0, h)$ if $\eta(0) = \eta_0$ and
\begin{equation}
\label{var-ineq}
\int_0^h  DE(\eta(t)) \langle \varphi(t) \rangle + D_2R(\eta(t), \partial_t \eta(t)) \langle \varphi(t)\rangle\,dt \ge \int_0^h \langle f(t), \varphi(t) \rangle_{L^2}\,dt
\end{equation}
holds for all $\varphi \in C([0,h];\Admr{\eta}{})$,\footnote{Here we use $L^{\infty}((0, h); \E)$ as a shorthand for the space of functions in $L^{\infty}((0, h); W^{2,p}(Q; \RR^n))$ that belong to $\E$ for a.e.\@ $t \in (0, h)$. Similarly, $C([0,h];\Admr{\eta}{})$ denotes the subset of $C([0, h]; W^{2, p}(Q;\RR^n))$ consisting of all functions with the property that $\varphi(t) \in \Admr{\eta(t)}{}$ for all $t \in [0, h]$.} where $\Admr{\eta}{}$ is the cone defined in \eqref{SID}.
\end{definition}

The existence of solutions to \eqref{aux-Pb} is established in the following theorem.

\begin{thm}
\label{AuxExistence}
Let $E$ satisfy \eqref{E1}--\eqref{E6}, $R$ satisfy \eqref{R1}, \eqref{R2}, \eqref{R3q} as well as \eqref{R4}, and let $h$, $\eta_0$, and $f$ be given as above. Then there exists a solution $\eta$ to \eqref{aux-Pb} in the sense of \Cref{sol-def-aux}. 
\end{thm}

\begin{proof}
We adopt the strategy employed in the initial two steps of the proof of \cite[Theorem 4.3]{Cesik2022}. To facilitate the comparison and highlight the differences, we retain the same structure.
\newline
\emph{Step 1:} Given $M \in \NN$, we set $\tau \coloneqq h/M$ and decompose $[0, h]$ into subintervals $[k \tau, (k + 1)\tau]$ of length $\tau$. Moreover, for every $1\leq k\leq M$ we let 
\begin{equation}
\label{time-avg-k}
f_k \coloneqq \frac{1}{\tau}\int_{(k - 1)\tau}^{k \tau}f(t)\,dt \in L^2(Q; \RR^n).
\end{equation}
We then define $\eta_k \in \E$ recursively via
\begin{equation}
\label{etak-def}
\eta_k \in \arg\min \left\{\J_k(\eta) : \eta \in \E  \right\},
\end{equation}
where
\begin{equation}
\label{Jk-def}
\J_k(\eta) \coloneqq E(\eta) + \tau R\left(\eta_{k - 1}, \tfrac{\eta - \eta_{k - 1}}{\tau}\right) - \tau \left \langle f_k, \tfrac{\eta - \eta_{k-1}}{\tau} \right\rangle_{L^2}.
\end{equation}
The existence of $\eta_k$ as in \eqref{etak-def} follows by a standard application of the direct method in the calculus of variations. Observe that since $\J_k(\eta_{k}) \leq \J_k(\eta_{k-1})$, expanding these inequalities and summing over $k = 1, \dots, m$, $m \le M$, yields a uniform estimate in the form of
\begin{equation} \label{eq:quasiAprioriEst}
 E(\eta_m) + \sum_{k=1}^{m} \tau R\left(\eta_{k-1},\tfrac{\eta_k-\eta_{k-1}}{\tau}\right) \leq E(\eta_0) + \tau \sum_{k=1}^m {\inner{f_k}{\tfrac{\eta_k-\eta_{k-1}}{\tau}}}_{L^2}.
\end{equation}
Next, for $t \in [(k - 1) \tau, k \tau)$, we let
\begin{equation}
\label{time-dep-interp}
\underline{\eta}_{\tau}(t) \coloneqq \eta_{k - 1}, \qquad \overline{\eta}_{\tau}(t) \coloneqq \eta_k, \qquad \text{ and }  \qquad \eta_{\tau}(t) \coloneqq \frac{k \tau - t}{\tau} \eta_{k - 1} + \frac{t - (k - 1) \tau}{\tau}\eta_k.
\end{equation}
Reasoning as in \cite[Theorem 4.3]{Cesik2022}, \eqref{eq:quasiAprioriEst} and the assumptions on $E$ and $R$ yield the existence of a subsequence (which we do not relabel) and a limiting deformation $\eta \in W^{1,2}((0, h); W^{1,2}(Q; \RR^n)) \cap L^\infty((0,h);\E)$ such that
\[
\arraycolsep=1.4pt\def\arraystretch{1.6}
\begin{array}{rll}
\underline{\eta}_{\tau} & \overset{\ast}{\rightharpoonup} {\eta} & \text{ in } L^{\infty}((0, h); W^{2,p}(Q;\RR^n)), \\
\overline{\eta}_{\tau} & \overset{\ast}{\rightharpoonup} {\eta} & \text{ in } L^{\infty}((0, h); W^{2,p}(Q;\RR^n)),
\end{array}
\]
and
\begin{equation}
\label{weakH1Hk}
\eta_{\tau} \rightharpoonup \eta \ \text{ in } W^{1,2}((0, h); W^{1,2}(Q; \RR^n)).
\end{equation}
\emph{Step 2:} A standard argument in the calculus of variations guarantees that each function $\eta_k$ (see \eqref{etak-def}) satisfies the Euler-Lagrange inequality
\begin{equation}
\label{EL-ineq-k}
0 \leq D\J_k(\eta_k)\langle \varphi \rangle
\end{equation}
for all $\varphi \in \Admr{\eta_k}{}$. Summing up these inequalities and replacing all terms with their time-dependent counterparts (see \eqref{time-dep-interp}), we obtain that
\begin{equation}
\label{approx-EL}
 0 \leq \int_0^h DE(\overline{\eta}_\tau(t)) \langle \varphi(t) \rangle  + D_2R(\underline{\eta}_\tau(t), \partial_t \eta_\tau(t)) \langle \varphi(t) \rangle - \inner{f_\tau(t)}{\varphi(t)} dt
\end{equation}
for all $\varphi \in L^{\infty}((0, h); W^{2, p}(Q; \RR^n))$ with $\varphi(t) \in \Admr{\overline{\eta}_\tau(t)}{}$ for $\mathcal{L}^1$-a.e.\@ $t \in (0,h)$. 
Therefore, in order to conclude the proof, we are left to show that all terms in \eqref{approx-EL} converge to their formal limit. For the last two, this is essentially a direct consequence of the weak convergence, together with \eqref{R4} and \Cref{prop:TErecoveryLocal} to ensure a strong approximation of test functions. We note here that in order to apply \Cref{prop:TErecoveryLocal}, we need to show that $\nabla \overline{\eta}_{\tau} \to \nabla \eta$ uniformly in $I \times Q$. This, in turn, can be obtained by observing that, in view of the definition, $\bar{\eta}_\tau(t) = \bar{\eta}_\tau(\lceil \tfrac{t}{\tau}\rceil \tau) = \eta_\tau(\lceil \tfrac{t}{\tau}\rceil \tau)$, and therefore 
\begin{align*}
 \norm[C^{1,\alpha}]{\bar{\eta}_\tau(t)-\eta(t)} \leq  \norm[C^{1,\alpha}]{\eta_\tau(\lceil \tfrac{t}{\tau}\rceil \tau)-\eta_\tau(t)}+ \norm[C^{1,\alpha}]{\eta_\tau(t)-\eta(t)}.
\end{align*}
Now, we have that the former term converges to $0$ by uniform continuity of the $\eta_\tau$, while the latter term converges to $0$ uniformly by the convergence in $C([0, h]; C^{1, \alpha}(Q; \RR^n))$ (which, up to the extraction of a further subsequence, follows by an application of the Aubin-Lions lemma).

For the potential energy however, due to the lack of regularization, we have to deviate from \cite{Cesik2022} and use a Minty-type argument that in \cite{Cesik2022} was only needed at a later stage. To this end, let $\psi \in C^{\infty}([0, h]; C^{\infty}_{\Gamma}(Q; [0, 1]))$ be given. Then, by \eqref{E6} and Lebesgue's dominated convergence theorem we see that 
\begin{align}
0 & \le \limsup_{\tau \to 0} \int_0^h [DE(\overline{\eta}_\tau(t)) - DE(\eta(t))]\langle (\overline{\eta}_\tau(t) - \eta(t)) \psi(t) \rangle\,dt \notag \\
& = \limsup_{\tau \to 0} \int_0^h DE(\overline{\eta}_\tau(t)) \langle (\overline{\eta}_\tau(t) - \eta(t)) \psi(t) \rangle\,dt. \label{minty-qs} & \notag
\\
& = \limsup_{\tau \to 0} \int_0^h DE(\overline{\eta}_\tau(t)) \langle \varphi_{\tau}(t) + \delta_{\tau} \tilde{t}_{\overline{\eta}_{\tau}}(t)\psi(t) \rangle\,dt,
\end{align}
where 
\[
\varphi_{\tau} \coloneqq (\overline{\eta}_\tau - \eta - \delta_{\tau}\tilde{t}_{\overline{\eta}_{\tau}}) \psi.
\]
We recall here that $\tilde{t}_{\overline{\eta}_{\tau}}$ is the uniformly interior vector field given by \Cref{smooth-unif-int}. Notice that $\delta_{\tau} \in \RR$ can be chosen in such a way that $-\varphi_{\tau} \in C([0,h];\Admr{\eta}{})$ and $\delta_{\tau} \to 0$ as $\tau \to 0$. Observe that the (approximate) Euler--Lagrange inequality \eqref{approx-EL} for $-\varphi_{\tau}$ can be rewritten as
\begin{equation}
\label{EL-varphitau}
\int_0^h DE(\overline{\eta}_\tau(t)) \langle \varphi_{\tau}(t) \rangle\,dt \le \int_0^h -D_2R(\underline{\eta}_\tau(t), \partial_t \eta_\tau(t)) \langle \varphi_{\tau}(t) \rangle + \inner{f_\tau(t)}{\varphi_{\tau}(t)} dt.
\end{equation}

Eventually extracting a subsequence (which we do not relabel), we have that 
\begin{equation}
\label{phitau-conv}
\varphi_{\tau}(t) \to 0 \qquad \text{ in } W^{1,2}(Q; \RR^n)
\end{equation} 
for $\mathcal{L}^1$-a.e.\@ $t$. Moreover, since $f_{\tau} \to f$ in $L^2((0, h); L^2(Q; \RR^n))$, an application of Lebesgue's dominated convergence theorem yields that
\begin{equation}
\label{fphito0}
\int_0^h \langle f_{\tau}(t), \varphi_{\tau}(t) \rangle \,dt \to 0
\end{equation}
as $\tau \to 0$. Since $\{D_2R(\underline{\eta}_{\tau}(t), \partial_t \eta_{\tau}(t))\}_{\tau}$ is bounded in $(W^{1,2})^*$ uniformly in $t$ (see \eqref{R4}), by \eqref{phitau-conv} we have that
\begin{equation}
\label{D2Rphito0}
\int_0^h D_2R(\underline{\eta}_{\tau}(t), \partial_t \eta_{\tau}(t))\langle \varphi_{\tau}(t) \rangle\,dt \to 0
\end{equation}
as $\tau \to 0$. Combining \eqref{minty-qs} with \eqref{EL-varphitau}, \eqref{fphito0}, and \eqref{D2Rphito0} we conclude that
\begin{align*}
\limsup_{\tau \to 0} \int_0^h [DE(\overline{\eta}_\tau(t)) - DE(\eta(t))] & \langle (\overline{\eta}_\tau(t) - \eta(t)) \psi(t) \rangle\,dt \\
& \le  \limsup_{\tau \to 0} \left|\int_0^h DE(\overline{\eta}_\tau(t)) \langle \delta_{\tau} \tilde{t}_{\overline{\eta}_{\tau}}(t)\psi(t) \rangle\,dt\right|  \\
& \le  \limsup_{\tau \to 0} \int_0^h \delta_{\tau} \|DE(\overline{\eta}_\tau(t))\|_{(W^{2, p})^*}\|\tilde{t}_{\overline{\eta}_{\tau}}(t)\psi(t)\|_{W^{2, p}}\,dt = 0.
\end{align*}
Since $\psi$ is arbitrary, we must have that 
\[
\limsup_{\tau \to 0} [DE(\overline{\eta}_\tau(t)) - DE(\eta(t))]\langle (\overline{\eta}_\tau(t) - \eta(t)) \psi(t) \rangle \le 0
\]
for a.e.\@ $t$. Consequently, \eqref{E6} implies that $\overline{\eta}_\tau(t) \to \eta(t)$ in $W^{2, p}(K; \RR^n)$ for all $K$ compactly contained in $\overline{Q}$ with $\dist(K, \Gamma) > 0$ and for $\mathcal{L}^1$-a.e.\@ $t$. Thus, we have shown that, by \eqref{E5}, $DE(\overline{\eta}_{\tau}(t)) \to DE(\eta(t))$ in $(W^{2,p}(Q; \RR^n))^*$. Since $\{DE(\overline{\eta}_{\tau}(t))\}_{\tau}$ is bounded in $(W^{2,p}(Q; \RR^n))^*$ uniformly in $t$, we claim that this is enough to prove the existence of a solution. Indeed, let $\varphi$ be as in \Cref{sol-def-aux}, then using \Cref{prop:TErecoveryLocal} we can find $\varphi_{\tau} \to \varphi$ that is admissible for the approximate Euler--Lagrange variational inequality (see \eqref{approx-EL}) and we can then let $\tau \to 0$.
\end{proof}

%

\subsection{Energy inequality}
We will derive the energy inequality using the previous construction and the so called Moreau-Yosida approximation. This approach, which is commonly employed in minimizing movements and can be traced back to De Giorgi, effectively circumvents any regularity issues associated with test functions by relying solely on the metric properties of energy and dissipation. Given that we will primarily utilize it with the time-delayed approximation, we will prove the energy inequality with the respective terms from the inertia. Importantly, this proof is applicable even when $\rho = 0$, that is, in the quasistatic case. Our approach draws inspiration from \cite[Sec. 3.1]{ambrosio2005}, although some modifications are necessary to accommodate forces and delve deeper into the discussion of inertial terms. On the other hand, the specific nature of our problem allows for certain simplifications, facilitating the analysis.

\begin{thm}[Energy inequality]\label{thm:td-energy-ineq}
The solutions constructed for equation \eqref{aux-Pb} satisfy the energy inequality
\begin{equation}
\label{EE-QS}
 E(\eta(s)) +  \int_0^s 2 R(\eta(t),\partial_t \eta(t)) \,dt  \leq E(\eta_0) + \int_0^s \inner[L^2]{f(t)}{\partial_t \eta(t)}\,dt
\end{equation}
for all $s \in [0, h]$. Similarly, for the time-delayed equation \eqref{aux-par}, we have that
\begin{multline}
\label{EE-TD}
 E(\eta(s)) + \int_0^s 2 R(\eta(t),\partial_t \eta(t)) \,dt + \frac{\rho}{2 h}\int_0^s \norm[L^2]{\partial_t \eta(t)}^2 \,dt \\ \leq E(\eta_0) + \frac{\rho}{2h}\int_0^s \norm[L^2]{\zeta(t)}^2 dt + \int_0^s \inner[L^2]{f(t)}{\partial_t \eta(t)}\,dt
\end{multline}
holds for all $s \in [0,h]$.
\end{thm}

\begin{proof}
Since the estimate in \eqref{EE-QS} is a special case of \eqref{EE-TD} with $\rho =0$, it suffices to prove the validity of the latter inequality.

Fix $\tau$ and $k \in \{1, \dots, M\}$. We define the Moreau-Yosida approximation by considering the family of functionals
\begin{align*}
 E_\theta(\eta;\eta_{k - 1}) &\coloneqq E(\eta) + \theta R\left( \eta_{k - 1},  \tfrac{\eta- \eta_{k - 1}}{\theta}\right)
 + \frac{\theta \rho}{2h} \norm[L^2]{\tfrac{\eta-\eta_{k - 1}}{\theta}-\zeta_k}^2 - \theta\inner[L^2]{f_k}{\tfrac{\eta-\eta_{k - 1}}{\theta}}
\end{align*}
for $\theta \in (0,\tau]$ as well as their minimizers:
\begin{align*}
\eta_\theta & \in \operatorname*{arg\,min}_{\eta \in \mathcal{E}} E_\theta(\eta;\eta_{k - 1}).
\end{align*}
Here $\zeta_k$ is used to denote the time average of $\zeta$ over the interval $[(k - 1)\tau, k \tau)$ (see \eqref{time-avg-k}). Reasoning as in the proof of \Cref{AuxExistence}, the existence of $\eta_\theta$ follows by an application of the direct method in the calculus of variations.

What we now want to do, is to vary $\theta$ between $0$ and $\tau$. In particular, we want to consider the (at the moment formal) equation
\begin{align}\label{eq:fundamentalTheorem}
 E_\tau(\eta_\tau; \eta_{k - 1}) - \lim_{\theta \searrow 0} E_\theta(\eta_\theta;\eta_{k - 1}) = \int_0^\tau \frac{d}{d\theta} E_\theta(\eta_\theta;\eta_{k - 1}) \, d\theta.
\end{align}
The reason for this is that, for $\theta = \tau$, without loss of generality we can identify $\eta_\tau$ with $\eta_{k}$ and get the energy and part of the other terms we need, while on the other hand $E_\theta(\eta_\theta; \eta_{k - 1})$ can be estimated by $E(\eta_{k - 1})$ for $\theta \searrow 0$. Finally, the integral on the right hand side will give the missing terms.

Let us begin with the limit $\theta \searrow 0$. Comparing the minimizer $\eta_\theta$ with the admissible competitor $\eta_{k - 1}$, we obtain that $E_\theta(\eta_{k - 1}; \eta_{k - 1}) \geq E_\theta(\eta_\theta;\eta_{k - 1})$ and thus
\begin{align*}
 E(\eta_{k - 1}) + \frac{\theta \rho}{2h} \norm[L^2]{\zeta_k}^2 
 \geq E_\theta(\eta_\theta; \eta_{k - 1}).
\end{align*}
This in turn immediately implies that 
\begin{equation}
\label{limsuptheta}
\limsup_{\theta \searrow 0} E_\theta(\eta_\theta;\eta_{k - 1}) \leq E(\eta_{k - 1}).
\end{equation}
Expanding the right hand side further, we also have
\[
E(\eta_{k - 1}) + \frac{\theta \rho}{2h} \norm[L^2]{\zeta_k}^2 \geq E(\eta_\theta) + \theta R\left( \eta_{k - 1},  \tfrac{\eta_{\theta} - \eta_{k - 1}}{\theta}\right)
 + \frac{\theta \rho}{2h} \norm[L^2]{\tfrac{\eta_\theta-\eta_{k - 1}}{\theta}-\zeta_k}^2 - \theta\inner[L^2]{f_k}{\tfrac{\eta_\theta-\eta_{k - 1}}{\theta}}.
\]
Multiplying the previous inequality by $\theta$, separating the dissipation-like terms of lower order, using Korn's inequality (see \eqref{R3}) and the $2$-homogenity of $R$, we get that
\begin{align*}
K_R \norm[W^{1,2}]{\eta_\theta - \eta_{k - 1}}^2 & \leq R( \eta_k,  \eta_\theta- \eta_{k - 1}) + \frac{\rho}{2h} \norm[L^2]{\eta_\theta-\eta_{k - 1}}^2 \\
  &\leq \theta\left(E(\eta_{k - 1})-E(\eta_\theta)\right) +  \frac{\theta \rho}{h} \inner[L^2]{\eta_\theta-\eta_{k - 1}}{ \zeta_k } + \theta\inner[L^2]{f_k}{\eta_\theta-\eta_{k - 1}} - \frac{\theta^2\rho}{h} \norm[L^2]{\zeta_k}^2.
\end{align*}
 Now, the scalar products can be estimated using Young's inequality and, for $\theta < \tau$ small enough, their parts involving $\eta_\theta$ can be absorbed on the left hand side. Furthermore, recall that $E$ is bounded from below, that $E(\eta_k)$ is uniformly bounded (see \eqref{eq:quasiAprioriEst}), and notice that we can drop the last term because of its sign. This then yields the uniform estimate
\begin{align*} 
 c \norm[W^{1,2}(Q)]{\eta_\theta - \eta_{k - 1}}^2  \leq \theta\left(E(\eta_{k - 1})-E_{\min} + \frac{\rho}{2h}\norm[L^2]{f_k}^2+ \frac{\rho}{2h}\norm[L^2]{\zeta_k}^2\right).
\end{align*}
In particular, from this we infer that $\eta_\theta \to \eta_{k - 1}$ for $\theta \searrow 0$.

Next, we compute the derivative $\frac{d}{d\theta} E_\theta(\eta_\theta;\eta_{k - 1})$. For this, let $\theta_1,\theta_2 \in (0,\tau]$ be given. Then a direct calculation yields that
\begin{multline*}
E_{\theta_1}(\eta_{\theta_1};\eta_{k - 1}) - E_{\theta_2}(\eta_{\theta_2};\eta_{k - 1}) \leq  E_{\theta_1}(\eta_{\theta_2};\eta_{k - 1}) - E_{\theta_2}(\eta_{\theta_2};\eta_{k - 1}) \\
= \left(\theta_1-\theta_2\right) \left[ -\frac{1}{\theta_1\theta_2} \left(R(\eta_{k - 1},\eta_{\theta_2}-\eta_{k - 1}) + \frac{\rho}{2h} \norm[L^2]{\eta_{\theta_2}-\eta_{k - 1}}^2  \right)  +  \frac{\rho}{2h}\norm[L^2]{\zeta_k}^2 \right].
\end{multline*}
Similarly, we also get
\begin{multline*}
E_{\theta_1}(\eta_{\theta_1};\eta_{k - 1}) - E_{\theta_2}(\eta_{\theta_2};\eta_{k - 1}) \geq  E_{\theta_1}(\eta_{\theta_1};\eta_{k - 1}) - E_{\theta_2}(\eta_{\theta_1};\eta_{k - 1}) \\ 
= \left(\theta_1-\theta_2\right) \left[ -\frac{1}{\theta_1\theta_2} \left(R(\eta_{k - 1},\eta_{\theta_1}-\eta_{k - 1}) + \frac{\rho}{2h} \norm[L^2]{\eta_{\theta_1}-\eta_{k - 1}}^2  \right) + \frac{\rho}{2h}\norm[L^2]{\zeta_k}^2 \right].
\end{multline*}
Dividing by $\theta_1-\theta_2$ then gives us an upper and a lower bound on the Lipschitz constant of $\theta \mapsto E_{\theta}(\eta_{\theta}; \eta_{k - 1})$, which is uniform whenever $\theta_1,\theta_2 > \theta_0$ for a fixed $\theta_0>0$, thus proving that this map is locally Lipschitz. Furthermore, sending $\theta_1\to \theta_2$ we get that
\begin{equation}
\label{Etheta-der}
 \frac{d}{d\theta} E_\theta(\eta_\theta;\eta_{k - 1}) = -\frac{1}{\theta^2}R(\eta_{k - 1},\eta_{\theta}-\eta_{k - 1}) - \frac{1}{\theta^2} \frac{\rho}{ 2h} \norm[L^2]{\eta_\theta-\eta_{k - 1}}^2  +\frac{\rho}{2h}\norm[L^2]{\zeta_k }^2 
\end{equation}
holds for almost all $\theta \in (0,\tau]$.
 
Next, combining \eqref{eq:fundamentalTheorem}, \eqref{limsuptheta}, and \eqref{Etheta-der} yields
\begin{multline*}
E(\eta_{k}) + \tau R\left( \eta_{k - 1},  \tfrac{\eta_{k}- \eta_{k - 1}}{\tau}\right)
 + \frac{\tau \rho}{2h} \norm[L^2]{\tfrac{\eta_{k}-\eta_{k - 1}}{\tau}-\zeta_k}^2- \tau\inner[L^2]{f_k}{\tfrac{\eta_{k}-\eta_{k - 1}}{\tau}} - E(\eta_{k - 1})  \\
 \leq - \int_0^\tau R\left(\eta_{k - 1},\tfrac{\eta_{\theta}-\eta_{k - 1}}{\theta}\right) + \frac{\rho}{2h} \norm[L^2]{\tfrac{\eta_\theta-\eta_{k - 1}}{\theta}}^2 d\theta +\tau \frac{\rho}{2h}\norm[L^2]{\zeta_k}^2. 
\end{multline*}
After reordering, we end up with
\begin{multline*}
E(\eta_{k}) + \int_0^\tau R\left( \eta_{k - 1},  \tfrac{\eta_{k}- \eta_{k - 1}}{\tau}\right) + R\left(\eta_{k - 1},\tfrac{\eta_{\theta}-\eta_{k - 1}}{\theta}\right)d\theta \\+ \frac{\rho}{2h} \int_0^\tau \norm[L^2]{\tfrac{\eta_{k}-\eta_{k - 1}}{\tau}-\zeta_k}^2 + \norm[L^2]{\tfrac{\eta_\theta-\eta_{k - 1}}{\theta}}^2 d\theta
 \leq E(\eta_{k - 1}) + \tau \frac{\rho}{2h}\norm[L^2]{\zeta_k }^2 +  \tau\inner[L^2]{f_k}{\tfrac{\eta_{k}-\eta_{k - 1}}{\tau}}.
\end{multline*}
This telescopes into an estimate over all of $[0,s]$ (where without loss of generality we can assume $s$ to be a multiple of $\tau$), which, after dropping the term including $\zeta_k$ on the left hand side, can be rewritten as
\begin{multline} 
\label{eq:approxResult}
E(\overline \eta_{\tau}(s)) + \int_0^{s} R(\underline \eta_{\tau}(t),\partial_t \eta_{\tau}(t)) + R(\underline \eta_{\tau}(t),\beta_{\tau}(t))\,dt + \frac{\rho}{2h} \int_0^{s} \norm[L^2]{\beta_{\tau}(t)}^2\,dt \\
\leq E(\eta_0) + \frac{\rho}{2h} \int_0^{s} \norm[L^2]{\zeta(t)}^2\, dt+ \int_0^s \inner[L^2]{f(t)}{\partial_t \eta_{\tau}(t)}\,dt
\end{multline}
where $\overline \eta_{\tau}, \underline \eta_{\tau}$, and $\eta_{\tau}$ are defined as in \eqref{time-dep-interp} and $\beta_{\tau}$ is the so called De Giorgi-interpolation, defined by
\[
 \beta_{\tau}(t) \coloneqq \frac{\eta_{\theta}-\eta_{k - 1}}{\theta} \qquad \text{ where } \theta \in (0,\tau], t = \tau k +\theta \text{ and } \eta_\theta \text{ is defined as above.}
\]
Now we send $\tau \to 0$ again. Then, as shown in the proof of \Cref{AuxExistence}, eventually extracting a subsequence we have that the linear and affine approximations, as well as the time-derivative of the affine approximation converge to $\eta$ and $\partial_t \eta$, respectively.

Next, we notice that \eqref{eq:approxResult} yields a uniform bound on $\beta_{\tau}$ in $L^2((0, h); W^{1,2}(Q; \RR^n))$. Therefore,
eventually extracting a further subsequence (which we do not relabel), we can assume that $\beta_{\tau}$ converges weakly to a limit $\beta$. The next step will be identifying the limit $\beta$ with $\partial_t \eta$.

For this we note that $\eta_\theta$ fulfills an Euler-Lagrange inequality similar to that for $\eta_{k+1}$. Using the previous definition this inequality reads as
\[
DE(\overline \eta_{\tau})\langle \varphi \rangle + D_2R(\underline \eta_{\tau}, \beta_{\tau})\langle \varphi \rangle + \frac{ \rho}{h}\inner[L^2]{\beta_{\tau}-\zeta_{\tau} }{\varphi} \geq \rho\inner[L^2]{f}{\varphi}  
\]
for almost all $t \in [0,h]$ and admissible test functions $\varphi$. We now integrate the previous inequality and send $\tau \to 0$. The only non-trivial convergence here is that of the first term, which we have already dealt with before. Then finally we end up with
\begin{align*}
&\int_0^{s} DE(\eta) \langle\varphi\rangle + D_2R\left(\eta,\beta\right)\langle\varphi\rangle +   \frac{ \rho}{h}\inner[L^2]{\beta-\zeta }{\varphi} dt \geq \int_0^{s} \rho\inner[L^2]{f}{\varphi}  dt
\end{align*}
for all admissible test functions. In particular this includes all compactly supported test functions. Thus, for these test functions, subtracting from this the equation for the time-delayed solution \eqref{aux-par} we get
\begin{align*}
&\int_0^{s} D_2R\left(\eta,\beta\right)\langle\varphi\rangle -D_2R\left(\eta,\partial_t \eta \right)\langle\varphi\rangle + \frac{ \rho}{h}\inner[L^2]{\beta-\partial_t \eta}{\varphi} dt = 0
\end{align*}
which implies that $\beta = \partial_t \eta$, since $D_2R(\eta, \cdot)$ is the derivative of a convex function and thus a monotone operator.

Finally, with this at hand and in view of the lower-semicontinuity of the terms on the left hand side of \eqref{eq:approxResult}, we obtain the desired energy inequality.
\end{proof}

\begin{rmk}[Energy estimates in the presence of corners] \label{rmk:energyEst}
 The previous proof significantly differs from its ``regular boundary'' counterpart in \cite{Cesik2022}. In that context, we derived a nearly identical estimate by simply testing with $-\partial_t \eta$. This process relied on satisfying two distinct admissibility criteria. 
 
 The first requirement was having sufficient regularity of $\partial_t \eta$, given that $DE(\eta)$ is a distribution of order $-2$ and that the bounds obtained from the dissipation $R(\eta,\partial_t \eta)$ only give control on the first derivative of $\partial_t \eta$. We achieved this by introducing an extra regularization term in the energy. The same strategy could be applied in this scenario.
 
 The other, more crucial requirement was that of $\partial_t \eta$ pointing in an admissible direction. This way, we were able to test the variational inequality, yielding the correct sign for the energy inequality. An equivalent approach would have been to test the equation involving the contact force with opportunely defined test functions and showing in this way that the additional forcing term has the right sign. Indeed, since the force always points in the interior normal direction and movement can never cause an overlap, this was precisely the case in \cite{Cesik2022}.
 
 However, for a less smooth boundary this approach no longer works that seamlessly, as evidenced by the example given in \Cref{rmk:testSpaces}. There, we arrived at a contact force pointing in a direction not blocked by any other solid, i.e.\@ there can be evolutions of the solid where the two corners simply glide past each other in either direction. Consequently, testing the contact measure with $\partial_t \eta$ can potentially produce terms of either sign. This again is the same reason why we can only recover an inequality for $\Admr{\eta}{}$ and not for the larger set $\Adm{\eta}{}$, as the above example would contradict the latter. Of course, it is natural to conjecture that having such a passing solution would be incompatible with generating such a force in the approximation. However, formalizing this argument is non-trivial and as we have shown, this complication can be circumvented by a more abstract approach.
\end{rmk}

\subsection{Existence of the contact force}
Our goal now is to refine the inequality in \eqref{var-ineq} by recasting it as an equation involving a contact force. To achieve this, we introduce the following definition.

\begin{definition}[Solution with a contact force]
\label{sol-def-aux-measure}
Let $\eta_0 \in \E$ and $f \in L^2((0, h); L^2(Q; \RR^n))$ be given. We say that the pair
\[
\eta \in W^{1,2}((0, h); W^{1,2}(Q; \RR^n)) \cap L^\infty((0, h); \E ), \qquad \sigma \in L^2_{w^*}((0,h);M(\partial Q; \RR^n)),
\]
where $\sigma$ is a contact force for $\eta$ (see \Cref{cf-def}), is a \emph{solution with a contact force} to \eqref{aux-Pb} if $\eta(0) = \eta_0$ and 
\begin{equation}
\label{eqn:measure-eqn}
\int_0^h \left[DE(\eta(t)) + D_2R(\eta(t), \partial_t \eta(t))\right] \langle \varphi(t)\rangle\,dt = \int_0^h \langle \sigma(t),\varphi(t)\rangle \,dt + \int_0^h \langle f(t), \varphi(t) \rangle_{L^2}\,dt
\end{equation}
for all $\varphi \in C([0,h]; W^{2,p}_{\Gamma}(Q;\RR^n))$.
\end{definition}

The existence of solutions with a contact force for problem \eqref{aux-Pb} is established in the following theorem.
\begin{thm}
\label{existence-CF-quasi}
Under the assumptions of \Cref{AuxExistence}, there exists a solution with a contact force to \eqref{aux-Pb}. Moreover, the solution satisfies the energy inequality
\[
E(\eta(t)) + 2 \int_0^t R(\eta(s), \partial_t \eta(s))\, ds \le E(\eta(0)) + \int_0^t \langle f(s), \partial_t \eta(s)\rangle\, ds
\]
for $t \in [0, h]$ and the estimate
\begin{equation}\label{cforce-h-estimate}
\int_0^h \|\sigma(t)\|_{M(\partial Q;\RR^n)}^2 dt \leq C_0 \left(h + \|\partial_t\eta\|_{L^2((0,h);W^{1,2}(Q))}^2 +\|f\|_{L^2((0,h)\times Q)}^2\right),
\end{equation}
where $C_0$ is a constant that depends only on $E(\eta_0)$.
\end{thm}

\begin{proof}
Here we continue to use the notation introduced in the proof of \Cref{AuxExistence}. In particular, we recall that $\eta_k$ (see  \eqref{etak-def}) is a minimizer for $\J_k$ (see \eqref{Jk-def}). We divide the proof into a several steps.
\newline
\emph{Step 1 (Lagrange multiplier is a measure):} In this first step we use the argument described in \cite[Theorem 3.1]{Palmer2018} to show that there is a contact force for $\eta_k$, namely $\sigma_k\in M(\partial Q;\RR^n)$, such that
\begin{equation}
\label{sigma-k-equation}
D\mathcal J_k(\eta_k)\langle \varphi \rangle - \langle \sigma_k,\varphi \rangle = 0
\end{equation}
for all $\varphi\in W^{2,p}_{\Gamma}(Q;\RR^n)$. We begin by letting
  \begin{align*}
M^-_{\eta_k} & \coloneqq \{(\ell,w) \in \RR \times C(\partial Q;\RR^n): \ell \leq 0 \text{ and } w(x) \in \interior \regtang{\eta_k}{x} \ \forall x \in \contact{\eta_k}\}, \\
M^+_{\eta_k} & \coloneqq \{ (\ell,w)\in\RR \times C(\partial Q;\RR^n): \exists \varphi \in \mathcal A_{\eta_k}^{w} \text{ s.t.\@ } D\mathcal{J}_k(\eta_k) \langle\varphi \rangle \leq \ell\},
\end{align*}
where for $w \in C(\partial Q;\RR^n)$
\begin{multline*}
\mathcal A_{\eta_k}^{w} \coloneqq \{\varphi \in W^{2,p}_{\Gamma}(Q;\RR^n) : \varphi(x) - \varphi(y) - w(x) + w(y) \in \interior \regtang{\eta_k}{x}-\interior \regtang{\eta_k}{y} \\ \text{ if } \eta_k(x) = \eta_k(y) \text{ with } x \neq y, \text{ and } \varphi(x)-w(x)\in\interior \regtang{\eta_k}{x} \text{ if } \eta_k(x)\in \partial \Omega \}.
\end{multline*}
As one can readily check, the sets $M^+_{\eta_k}$ and $M^-_{\eta_k}$ are convex as a consequence of the convexity of the regular tangent cones. Moreover, $M_{\eta_k}^{-}\neq \emptyset$ by \Cref{smooth-unif-int} $(i)$ and has nonempty interior in view of \Cref{regtang-char}.

Next, we claim that $M^+_{\eta_k}\cap \interior M^-_{\eta_k} = \emptyset$. Indeed, notice that if $(\ell,w) \in \interior M_{\eta_k}^-$ then
$\mathcal A_{\eta_k}^w \subset \Admr{\eta}{}$, and therefore (see \eqref{EL-ineq-k}) we have that
\[
D\mathcal{J}_k(\eta_k) \langle \varphi \rangle \ge 0 > \ell.
\]
This shows that $(\ell,w) \notin M_{\eta_k}^+$, thus proving the claim. Therefore, by the Hahn-Banach theorem there exists a separating hyperplane, that is, a pair $(0,0)\neq(\lambda_0,\sigma_0)\in \RR\times M(\partial Q;\RR^n)$ such that
\begin{align*}
\lambda_0 \ell+\left\langle\sigma_0, w\right\rangle & \geq 0, \qquad \text{ for all } (\ell, w) \in M^{-}_{\eta_k}, \\
\lambda_0 \ell+\left\langle\sigma_0, w\right\rangle & \leq 0, \qquad 
\text{ for all } (\ell, w) \in M_{\eta_k}^{+}.
\end{align*}
Notice that if $w \in C(\partial Q;\RR^n)$ is such that $C_{\eta_k} \cap \supp w = \emptyset$, then $(0,w) \in M_{\eta_k}^+\cap M_{\eta_k}^-$ and therefore $\langle \sigma_0, w\rangle = 0$. This shows $\supp \sigma_0 \subset C_{\eta_k}$. Thus, by \Cref{cforce-char} we have that $\sigma_0$ is a contact force for $\eta_k$.

Let $\xi_{\Gamma}$ be given as in \Cref{prop:TErecoveryLocal} and notice that $(2 D \mathcal J_k(\eta_k)\langle \tilde t_{\eta_k} \xi_{\Gamma} \rangle , \tilde t_{\eta_k}\xi_{\Gamma}) \in M^+_{\eta_k}$ by the definition of $M^{+}_{\eta_k}$.
Therefore
\begin{equation}
\lambda_0 2 D \mathcal J_k(\eta_k)\langle \tilde t_{\eta_k} \xi_{\Gamma} \rangle + \langle \sigma_0, \tilde t_{\eta_k} \xi_{\Gamma} \rangle \le 0.
\end{equation}
From this we see that $\lambda_0 < 0$. Indeed, since $\lambda_0 = 0$ implies $\sigma_0 \neq 0$ (by the Hahn-Banach theorem) and since $\langle \sigma_0,  \tilde t_{\eta_k}\xi_{\Gamma} \rangle > 0$ in view of the fact that $\tilde t_{\eta_k}$ is a uniformly interior field, then necessarily we must have that $\lambda_0 \neq 0$. On the other hand, if $\lambda_0 > 0$ then by taking $\ell > \max(0,2 D \mathcal J_k(\eta_k)\langle \tilde t_{\eta_k} \xi_{\Gamma}\rangle)$ we get that $(\ell, \tilde t_{\eta_k}\xi_{\Gamma}) \in M^+_{\eta_k}$. However, in this case we readily reach a contradiction by noticing that $\lambda_0 \ell + \langle \sigma_0,  \tilde t_{\eta_k}\xi_{\Gamma} \rangle > 0$.
 
Finally, set $\sigma_k \coloneqq -\sigma_0/\lambda_0$. Then $\sigma_k$ is a contact force for $\eta_k$ and a simple computation shows that equation \eqref{sigma-k-equation} is satisfied.
\newline
\emph{Step 2 (bound on the measure):}
For $\sigma_k$ as above, let $\sigma_{\tau}$ be defined via
\begin{equation}
\sigma_\tau(t) \coloneqq \sigma_k,\qquad t\in [\tau (k-1), \tau k).
\end{equation}
By the assumptions on $\eta_0|_\Gamma$ and the local injectivity of $\eta$, we can pick a compact subset $K\subset \partial Q$ such that $\eta(t)|_{\partial Q \setminus K}$ is always injective. Then by the action-reaction principle, estimating $\|\sigma_k \|_{M(K;\RR^n)}$ is enough to estimate $\|\sigma_k\|_{M(\partial Q;\RR^n)}$. Wlog.\ we can choose $\xi_\Gamma$ from before such that $\xi_\Gamma(x) = 1$ for all $x \in K$.
Let us use $\xi_\Gamma\tilde t_{\eta_k}$ with $\tilde t_{\eta_k}\in C^{2}(\overline Q;\RR^n)$ from \Cref{smooth-unif-int} $(i)$ as a test function.  This gives us (as $d\sigma_k = g_k\,d|\sigma_k|$, see \Cref{cf-def}), denoting $\alpha\coloneqq \sin \vartheta$ from \Cref{unif-int-def},
\begin{equation}
-D\mathcal J_k(\eta_k)\langle \xi_\Gamma \tilde t_{\eta_k}\rangle = \langle \sigma_k, \xi_\Gamma\tilde t_{\eta_k}\rangle  =  \int_{\partial Q} \xi_\Gamma \tilde{t}_{\eta_k}\cdot g_k d|\sigma_k| \geq \alpha \int_{K} |\tilde t _{\eta_k}| |g_k| d|\sigma_k| = \alpha \|\sigma_k\|_{M(K;\RR^n)}.
\end{equation}
Thus	
\begin{equation}
\|\sigma_k\|_{M(\partial Q;\RR^n)} \leq -\frac{2}{\alpha} \left( DE(\eta_k)\langle \tilde t _{\eta_k}\rangle +DR\left(\eta_{k-1}, \frac{\eta_k-\eta_{k-1}}{\tau} \right) \langle \tilde t_{\eta_k}\rangle - \langle f_k,\tilde t_{\eta_k}\rangle \right)
\end{equation}
Multiplying this by $\psi\in C([0,h];\RR^+)$ and integrating over $t \in (0,h)$
\begin{multline*}
\int_0^h \|\sigma_\tau(t)\|_{M(\partial Q;\RR^n)} \psi(t) dt \\ \leq -\frac{2}{\alpha} \int_0^h \left( DE(\overline \eta_\tau(t))\langle \tilde t_{\overline \eta_\tau}(t) \rangle +DR\left(\underline \eta_\tau(t),  \partial_t \eta_\tau (t) \right) \langle \tilde t_{\overline \eta_\tau}(t)\rangle - \langle \overline f_\tau(t),\tilde t_{\overline \eta_\tau}(t)\rangle  \right)\psi(t)\,dt.
\end{multline*}
Let us estimate all terms on the right hand side. First,
\begin{equation}
\left|\int_0^h DE(\overline \eta_\tau(t))\langle \tilde t_{\overline \eta_\tau}(t) \rangle \psi(t)\, dt \right|\leq \|DE(\overline \eta_\tau)\|_{L^\infty((W^{2,p})^*)} \| \tilde t_{\overline \eta_\tau}\|_{L^\infty(W^{2,p})} \sqrt h \|\psi\|_{L^2((0,h))}
\end{equation}
and by the energy inequality \Cref{thm:td-energy-ineq} and by \eqref{E5}, we have $\|DE(\overline \eta_\tau)\|_{L^\infty((W^{2,p})^*)}$ bounded by a constant $C_{E_0}$ depending only on $E_0$.
Next, using \eqref{eq:D2Rgrowthbounds} we get
\begin{align}
\left|\int_0^h DR\left(\underline \eta_\tau(t), \partial_t \eta_\tau (t) \right) \langle \tilde t_{\overline \eta_\tau}(t)\rangle \psi(t)\, dt \right| & \leq \int_0^h  C_{E_0}\|\partial_t \eta_\tau(t)\|_{W^{1,2}} \|\tilde t _{\overline \eta_\tau}(t)\|_{W^{1,2}} \psi(t) \,dt \\
& \leq C_{E_0} \|\partial_t\eta_\tau\|_{L^2(W^{1,2})}\|\tilde t_{\overline\eta_\tau}\|_{L^\infty(W^{1,2})} \| \psi\|_{L^2((0,h))}
\end{align}
and finally 
\begin{equation}
\int_0^h - \langle \overline f_\tau(t),\tilde t_{\overline \eta_\tau}(t)\rangle\psi(t)\, dt \leq \|f\|_{L^2(L^2)} \| \tilde t_{\overline \eta_\tau}\|_{L^\infty(L^2)} \| \psi\|_{L^2((0,h))}.
\end{equation}
So altogether, given that by \Cref{smooth-unif-int} $(i)$ $\tilde t_{\overline \eta_\tau}$ is bounded in all the norms above by a constant, that 
\begin{equation}
\int_0^h \|\sigma_\tau(t)\|_{M(\partial Q;\RR^n)} \psi(t) dt \leq  \frac{1}{\alpha} C_{E_0} (\sqrt h + \|\partial_t\eta_\tau\|_{L^2(W^{1,2})}+ \|f\|_{L^2(L^2)})\|\psi\|_{L^2((0,h))}
\end{equation}
Taking a supremum over $\psi$ such that $\|\psi\|_{L^2((0,h))}\leq 1$ yields

\begin{equation}\label{sigma-tau-estimate}
\|\sigma_\tau\|_{L^2((0,h);M(\partial Q;\RR^n))} \leq \frac{1}{\alpha} C_{E_0}( \sqrt h + \|\partial_t\eta_\tau\|_{L^2((0,h);W^{1,2}(Q))} + \|f\|_{L^2((0,h)\times Q)}).
\end{equation}
Hence, we have proved that $\sigma_\tau\in L^2((0,h);M(\partial Q;\RR^n))$ and moreover the bound is independent of $\tau$, since by \Cref{thm:td-energy-ineq} $\|\partial_t\eta_\tau\|_{L^2((0,h);W^{1,2}(Q))}$ is bounded independently of $\tau$.

\emph{Step 3 (passing to the limit):}
We have from Step 1 that 
\begin{equation*}
\int_0^h DE(\overline \eta_\tau) \langle \varphi \rangle + DR(\underline\eta_\tau ,\partial_t\eta_\tau) \langle \varphi\rangle  - \langle \sigma_\tau , \varphi \rangle \,dt = \int_0^h \langle f_\tau, \varphi \rangle \,dt ,\quad \varphi\in W^{2,p}_{\Gamma}(Q;\RR^n).
\end{equation*}
By the estimate \eqref{sigma-tau-estimate}, using the compactness of contact forces from \Cref{compactness-closure-time} we have that for a subsequence of $\tau \to 0$ (which we do not relabel) we have 
\begin{equation*}
\sigma_\tau \weakstarto \sigma \text{ in } M([0,h]\times \partial Q),
\end{equation*}
where $\sigma$ is a contact force for $\eta$. Moreover it satisfies the equation (passage to the limit in all of the other terms has already been established in the proof of \Cref{AuxExistence})
\begin{equation*}
\int_0^h DE(\eta) \langle \varphi \rangle + DR(\eta,\partial_t\eta) \langle \varphi \rangle - \langle \sigma , \varphi \rangle \,dt = \int_0^h \langle f, \varphi \rangle \,dt
\end{equation*}
by a density argument for all $\varphi\in C([0,h];W^{2,p}_{\Gamma}(Q;\RR^n))$. Moreover by \eqref{sigma-tau-estimate} and \Cref{cforce-l2intime} we have also $\sigma\in L^2_{w^*}((0,h);M(\partial Q;\RR^n))$ with the same estimate as \eqref{sigma-tau-estimate}, that is
\begin{equation*}
\|\sigma\|_{L^2_{w^*}((0,h);M(\partial Q;\RR^n))} \leq \frac{1}{\alpha} C_{E_0}\left(\sqrt h + \|\partial_t\eta\|_{L^2((0,h);W^{1,2}(Q))} + \|f\|_{L^2((0,h)\times Q)}\right),
\end{equation*}
finishing the proof.
\end{proof}

\section{Proof of the main result}
\label{sec:inertial}
In this section, we adapt the several steps in the proof of Theorem 2.4 in \cite[Section 5]{Cesik2022} to our more general situation. Thus, we will again only sketch the arguments that are identical in the case of regular boundaries and mainly focus on the differences. To make the comparison easier, we keep the same structure.

\begin{proof}[Proof of \Cref{main-thm-VI}] 
\emph{Step 1:} Recall that by using a different approach for the energy inequality, we were able to obtain the results of the previous section without relying on additional regularization terms. Therefore, in the current framework, there is no need to regularize the initial data. We can thus skip this step.
\newline
\emph{Step 2:} For fixed $h$, on each interval $[lh,(l+1)h]$, for $l\in \NN$, we now iteratively apply \Cref{AuxExistence} with
\begin{equation}\label{Rtilde-ftilde}
 \tilde{R}^{(h)}(\eta, b) \coloneqq R(\eta, b) + \frac{\rho}{2h}\|b\|^2_{L^2} \qquad \text{ and } \qquad \tilde{f}(t) \coloneqq f(t) + \frac{\rho}{h}\partial_t \eta(t - h)
\end{equation}
in place of $R$ and $f$. Here we implicitly set $\partial_t \eta(t) \coloneqq \eta^*$ for $t \leq 0$. 

Note that on each interval, the energy inequality for the previous interval gives us a control on $\partial_t \eta$ which guarantees that $\tilde{f}$ is well defined in the next interval. It is immediate to see that if $R$ is given as in \eqref{R3}, then $\tilde{R}^{(h)}$ satisfies \eqref{R3q}; in particular, this implies that all the assumptions of \Cref{AuxExistence} are satisfied for all $l$.

Next, we observe that the piecewise constant function defined on the entire interval $[0, T]$ obtained by gluing the solutions on the sub-intervals $[lh,(l+1)h]$, namely $\eta^{(h)}$, is a weak solution to
\begin{equation} 
\label{eq:tdEquation} 
\rho \frac{\partial_t \eta^{(h)}(t) - \partial_t \eta^{(h)}(t-h)}{h} + DE(\eta^{(h)}(t))+D_2 R(\eta^{(h)}(t), \partial_t\eta^{(h)}(t)) = f(t),
\end{equation}
for each $h>0$.
\newline
\emph{Step 3:} A closer look at the energy inequalities (see \Cref{thm:td-energy-ineq}) for each interval $[lh,(l+1)h]$ reveals that the terms for kinetic and potential energy on each side of the equation cancel when summed up. In turn, we find that
\begin{multline} 
\label{eq:Fulltd-energy-ineq}
  E(\eta^{(h)}(s)) + \fint_{s-h}^s \frac{\rho}{2} \norm[L^2]{\partial_t \eta^{(h)}(t)}^2 dt + \int_0^s 2 R(\eta^{(h)}(t),\partial_t \eta^{(h)}(t))\,dt \\ \leq E(\eta_0) + \frac{\rho}{2}\norm[L^2]{\eta^*}^2  + \int_0^s \inner[L^2]{f(t)}{\partial_t \eta^{(h)}(t)}\,dt
\end{multline}
holds for all $s \in [0,T]$.  From this, but also using the properties of $E$ and $R$, one easily derives uniform bounds (with respect to $h$) for
\begin{equation}
 \eta^{(h)} \in L^\infty((0,T);W^{2,p}(Q;\RR^n)) \qquad \text{ and } \qquad \partial_t \eta^{(h)} \in  L^2((0,T);W^{1,2}(Q;\RR^n)).
\end{equation}
In particular, this implies the existence of a limit deformation $\eta$ and a subsequence such that
\begin{equation}
\label{etah-to-eta-1}
\arraycolsep=1.4pt\def\arraystretch{1.6}
\begin{array}{rll}
\eta^{(h)} & \overset{\ast}{\rightharpoonup} \eta & \text{ in } L^{\infty}((0, T); W^{2, p}(Q; \RR^n)), \\
\eta^{(h)} & \rightharpoonup \eta & \text{ in } W^{1,2}((0, T); W^{1,2}(Q; \RR^n)).
\end{array}
\end{equation}
\emph{Step 4:} This step is dedicated to improving the convergences in \eqref{etah-to-eta-1}. This is necessary in order to pass to the limit as $h \to 0$ in the term that contains $DE(\eta^{(h)})$ in the weak formulation for \eqref{eq:tdEquation}. To this end, we first consider the auxiliary function
\begin{equation*}
 b^{(h)}(t) \coloneqq \fint^{t+h}_{t} \partial_t \eta^{(h)}(s)\,ds. 
\end{equation*}
From the bounds on $\partial_t \eta^{(h)}$ we quickly derive that $b^{(h)} \in L^2((0,T);W^{1,2}(Q;\RR^n))$. Additionally, noting that using \eqref{eq:tdEquation} we have that
\begin{equation}\label{dtb-negativespace}
 \partial_t b^{(h)}(\cdot) = \frac{\partial_t \eta^{(h)}(\cdot+h)-\partial_t \eta^{(h)}(\cdot)}{h} \in \left(L^2((0,T); W^{2,p}_0(Q;\RR^n)) \right)^*
\end{equation}
admits a uniform bound, we can use the Aubin-Lions theorem to prove that $b^{(h)}$ converges strongly (eventually extracting a further subsequence) in $L^2((0,T) \times Q;\RR^n)$ and we can further identify its limit with $\partial_t \eta$.

Now note that using our definitions for $\tilde{R}^{(h)}$ and $\tilde{f}$, the solution $\eta^{(h)}$ satisfies
 \begin{multline} \label{weaketah7}
\int_0^T DE(\eta^{(h)}(t))\langle  \varphi(t) \rangle + D_2R(\eta^{(h)}(t), \partial_t \eta^{(h)}(t))\langle  \varphi(t)\rangle\,dt - \rho \langle \eta^*, \varphi(0)\rangle_{L^2} \\ + \int_0^T \rho \left\langle \tfrac{\partial_t \eta^{(h)}(t) -\partial_t \eta^{(h)}(t-h)}{h}, \varphi(t) \right\rangle_{L^2} \,dt = \int_{[0,T]\times \partial Q} \varphi(t,x)\cdot d \sigma(t,x) + \int_0^T \langle f(t), \varphi(t) \rangle_{L^2}\,dt
 \end{multline}
for all $\varphi \in C([0, T]; W^{2, p}_{\Gamma}(Q; \RR^n))$.
 
 Now let $\tilde t_{\eta^{(h)}}$ be given as in \Cref{smooth-unif-int}, $\xi_{\Gamma} \in C^\infty(Q; [0,1])$ be a cutoff that satisfies $\xi_{\Gamma}(x) = 0$ for $x \in \Gamma$ and $\xi_{\Gamma}(x) > 0$ for $x \in \overline{Q} \setminus \overline{\Gamma}$, and let $\psi \in L^\infty((0,T))$. We then use $\xi_{\Gamma} \tilde t_{\eta^{(h)}} \psi$ as a test function in \eqref{weaketah7}. The rest of this step then proceeds precisely as in \cite{Cesik2022} and we will thus only highlight some of the details.

Note that $\tilde t_{\eta^{(h)}} \in L^\infty((0,T); C^{2}(\overline Q;\RR^n))$. From this we obtain 
\begin{multline}
\int_0^T \|\sigma^{(h)}(t)\|_{M(\partial Q;\RR^n)}\psi(t) dt \leq \frac{1}{\alpha}\int_0^T\biggr\{\left[DE(\eta^{(h)}(t)) + D_2R(\eta^{(h)}(t), \partial_t \eta^{(h)}(t))\right] \langle \xi_{\Gamma} \tilde t_{\eta^{(h)}}(t)\rangle \\ + \rho \left \langle \tfrac{\partial_t \eta^{(h)}(t)-\partial_t\eta^{(h)}(t-h)}{h}, \xi_{\Gamma} \tilde t_{\eta^{(h)}}(t)\right \rangle_{L^2} - \langle f(t), \xi_{\Gamma} \tilde t_{\eta^{(h)}}(t) \rangle_{L^2}\biggr\}\psi(t)\,dt. \label{toestimate}
\end{multline}
Taking $\psi = 1$, we can then proceed exactly as in the corresponding step of Theorem 2.4 in \cite{Cesik2022} and estimate all the terms on the right-hand side of \eqref{toestimate}. We recall here that the difference quotient used to approximate the inertial term can be rewritten as
\begin{multline} \label{eq:inertialForceEst}
  \int_0^T \rho \left \langle \tfrac{\partial_t \eta^{(h)}(t)-\partial_t\eta^{(h)}(t-h)}{h}, \xi_{\Gamma} \tilde t_{\eta^{(h)}}(t)\right \rangle_{L^2}dt = 
 \int_0^{T-h} \rho \left \langle \partial_t \eta^{(h)}(t), \xi_{\Gamma} \tfrac{\tilde t_{\eta^{(h)}}(t)-\tilde t_{\eta^{(h)}}(t+h)}{h}\right \rangle_{L^2}dt \\ + \int_{T - h}^T \frac{\rho}{h} \langle \partial_t \eta^{(h)}(t), \xi_{\Gamma} \tilde t_{\eta^{(h)}}(t) \rangle_{L^2}\,dt - \int_0^h \frac{\rho}{h} \langle \partial_t \eta^{(h)}(t - h), \xi_{\Gamma} \tilde t_{\eta^{(h)}}(t) \rangle_{L^2}\,dt,
\end{multline}
and that the first term on the right-hand is controlled by the bounds on $\partial_t \tilde{t}_{\eta^{(h)}}$ obtained in \Cref{smooth-unif-int}. Together with the other estimates, this implies that $\sigma^{(h)} \in L^1_{w^*}((0,T); M(\partial Q; \RR^n))$. Consequently, \Cref{compactness-closure-time} guarantees the existence of a subsequence (which we do not relabel) such that $\sigma^{(h)}\weakstarto \sigma$ in $M([0,T] \times \partial Q;\RR^n)$, where $\sigma \in M([0,T] \times \partial Q;\RR^n)$ is a contact force for $\eta$ and satisfies the action-reaction principle.

Next we consider the case where $\psi$ is the characteristic function of a small interval $[t_0,t_0+\delta]$. Observe that $\tilde{t}_{\eta^{(h)}}$ can be constructed in such a way that it is piecewise constant in time and that with this choice, reasoning as in \eqref{eq:inertialForceEst}, we have that 
\begin{multline*}
 \int_{t_0}^{t_0+\delta} \rho \left \langle \tfrac{\partial_t \eta^{(h)}(t)-\partial_t\eta^{(h)}(t-h)}{h}, \xi_{\Gamma} \tilde t_{\eta^{(h)}}(t)\right \rangle_{L^2} dt \\
= \fint_{t_0+\delta-h}^{t_0+\delta} \rho \left \langle \partial_t \eta^{(h)}(t), \xi_{\Gamma} \tilde t_{\eta^{(h)}}(t)\right \rangle_{L^2} dt 
- \fint_{t_0-h}^{t_0} \rho \left \langle \partial_t \eta^{(h)}(t), \xi_{\Gamma} \tilde t_{\eta^{(h)}}(t)\right \rangle_{L^2} dt
\end{multline*}
which can easily be bounded uniformly by the energy-inequality and in fact arbitrarily small, by limiting the support of $\tilde{t}_{\eta^{(h)}}$ to a smaller neighborhood of $\partial Q$. Together with the fact that the estimates of the other terms depend on $\norm[L^2]{\psi}$ and thus vanish when $\delta \to 0$, this implies the absence of any concentrations in time. Again, we refer the reader to \cite{Cesik2022} for more details, as the proof of these estimates is unchanged.
%
\newline
\emph{Step 5:} With this in hand, using \eqref{E6}, we can conclude that 
\begin{equation*}
 0 \leq \liminf_{h\to 0} \int_0^T [DE(\eta^{(h)}) - DE(\eta)]\langle (\eta^{(h)}-\eta) \psi \rangle\,dt
\end{equation*}
for all $\psi \in C^{\infty}_{\Gamma}(Q; [0, 1])$. Again, as in \Cref{AuxExistence}, we do not need to deal with the regularization, making this argument more straightforward when compared to our proof in \cite{Cesik2022}. The only additional term we get after integrating by parts is
\begin{align*}
 - \liminf_{h\to 0} \int_0^T \inner[L^2]{\tfrac{\partial_t \eta^{(h)}(\cdot+h)-\partial_t \eta^{(h)}(\cdot)}{h}}{\psi (\eta^{(h)}-\eta) } dt = \liminf_{h \to 0} \int_0^T \inner[L^2]{b^{(h)}}{\psi \partial_t (\eta^{(h)}-\eta) }\,dt,
\end{align*}
which converges to zero. Moreover, we have that $\partial_t \eta^{(h)} \rightharpoonup \partial_t \eta$ and $b^{(h)} \to \partial_t \eta$, respectively, in $L^2((0,T) \times Q;\RR^n)$. Combining these facts with \eqref{E6} we then conclude that $\eta^{(h)}(t) \to \eta(t)$ in $W^{2, p}(K; \RR^n)$ for all $K$ compactly contained in $\overline{Q}$ with $\dist(K, \Gamma) > 0$ and for almost all $t \in [0, T]$.
\newline
\emph{Step 6:} The previously shown convergences now allow us to pass to the limit with all terms in the equation, including the one involving the contact force. This readily implies the existence of a solution to the full problem with a contact force and therefore completes the proof.
\end{proof}

\bibliographystyle{alpha}
\bibliography{biblio}

\begin{thebibliography}{{\v{C}}GK22}

\bibitem[AGS05]{ambrosio2005}
Luigi Ambrosio, Nicola Gigli, and Giuseppe Savaré.
\newblock {\em Gradient Flows in Metric Spaces and in the Space of Probability
  Measures}.
\newblock Lectures in Mathematics. ETH Z{\"u}rich. Birkhäuser Basel, 2005.

\bibitem[BKS23]{Benesova2020}
Barbora Bene{\v s}ov{\'a}, Malte Kampschulte, and Sebastian Schwarzacher.
\newblock A variational approach to hyperbolic evolutions and fluid-structure
  interactions.
\newblock {\em J. Eur. Math. Soc. (online first)}, 2023.

\bibitem[{\v{C}}GK22]{Cesik2022}
Antonín {\v{C}}e{\v{s}\'i}k, Giovanni Gravina, and Malte Kampschulte.
\newblock Inertial evolution of non-linear viscoelastic solids in the face of
  (self-)collision.
\newblock arXiv:2212.00705, December 2022.

\bibitem[Cla87]{Clarke1987}
Frank~H. Clarke.
\newblock {\em Optimization and Nonsmooth Analysis (Classics in Applied
  Mathematics)}.
\newblock Society for Industrial Mathematics, 1987.

\bibitem[CN87]{Ciarlet1987}
Philippe~G. Ciarlet and Jind{\v{r}}ich Ne{\v{c}}as.
\newblock Injectivity and self-contact in nonlinear elasticity.
\newblock {\em Archive for Rational Mechanics and Analysis}, 97(3):171--188,
  September 1987.

\bibitem[HK09]{Healey2009}
Timothy~J. Healey and Stefan Krömer.
\newblock Injective weak solutions in second-gradient nonlinear elasticity.
\newblock {\em ESAIM: Control, Optimisation and Calculus of Variations},
  15(4):863--871, July 2009.

\bibitem[KR20]{kromerQuasistaticViscoelasticitySelfcontact2019}
Stefan Kr\"{o}mer and Tom\'{a}\v{s} Roub\'{\i}\v{c}ek.
\newblock Quasistatic viscoelasticity with self-contact at large strains.
\newblock {\em J. Elasticity}, 142(2):433--445, 2020.

\bibitem[Pal18]{Palmer2018}
Aaron~Zeff Palmer.
\newblock Variations of deformations with self-contact on lipschitz domains.
\newblock {\em Set-Valued and Variational Analysis}, 27(3):807--818, June 2018.

\bibitem[RW98]{Rockafellar1998}
R.~Tyrrell Rockafellar and Roger J.~B. Wets.
\newblock {\em Variational Analysis}.
\newblock Springer Berlin Heidelberg, 1998.

\bibitem[Sch02]{Schuricht2002}
Friedemann Schuricht.
\newblock Variational approach to contact problems in nonlinear elasticity.
\newblock {\em Calculus of Variations and Partial Differential Equations},
  15(4):433--449, December 2002.

\end{thebibliography}
\end{document}